%% file: main.tex
\documentclass[11pt]{article}
\input{Macros}

\newcommand{\authorinfo}{%
\par\bigskip
\begingroup
\footnotesize

\noindent
\textsc{Liang Chang}, School of Mathematical Sciences and LPMC, Nankai University, Tianjin, China\\
\textit{Email address:}
\texttt{changliang996@naankai.edu.cn}

\medskip
\noindent 
\textsc{Siu-Hung Ng}, Department of Mathematics, Baton Rouge, LA 70803, USA\\
\textit{Email address:}
\texttt{rng@math.lsu.edu}

\medskip
\noindent 
\textsc{Yilong Wang}, Beijing Institute of Mathematical Sciences and Applications (BIMSA), Beijing, China\\
\textit{Email address:}
\texttt{wyl@bimsa.cn}

\endgroup
}
\title{\textbf{Modular fusion categories with trivial Torelli group actions}}
\author{Liang Chang, Siu-Hung Ng, Yilong Wang}
\date{}
\begin{document}

\maketitle

\begin{abstract}
In this paper, we study modular fusion categories whose mapping class group representations are trivial on the Torelli groups, with particular emphasis on the congruence properties of the resulting representations of $\Sp(2g,\bZ)$. We prove that, for any $g \ge 3$, the Torelli group is contained in the kernel of the genus-$g$ mapping class group representation $\rho_g$ associated with a modular fusion category $\CC$ if and only if $\CC$ is pointed. This refines the corresponding result in \cite{MW25}. We also characterize the modular fusion categories for which the genus-$2$ Torelli group is contained in $\ker \rho_2$; in particular, categories with isotropic adjoint subcategories belong to this class. Finally, we give a complete classification of modular fusion categories with isotropic adjoint subcategories, and determine the groups $G$ satisfying $\zdr(\Vec_G)$ is 2-Torelli trivial but $\zdr(\Vec_G)_{\ad}$ is not isotropic are precisely 2-Engel groups of nilpotent class 3.
\end{abstract}
\tableofcontents
%========================================
% Intro
%========================================
\section{Introduction}
Modular fusion categories give rise to Reshetikhin-Turaev topological  field theories, which produces a projective representation of the mapping class group $\MCG{g}$ of a closed oriented surface $\Sigma_{g}$ of genus $g$ \cite{Tur10}. When $g=1$, $\MCG{1} \cong \SL(2, \bZ)$, and the algebraic properties of the TFT representations of $\SL(2,\bZ)$, such as the congruence property and the Galois symmetry, are extensively studied in the literature \cite{dBG, CG94, SZ12, NS10, DLN15}. These algebraic properties not only provide important insights into the study of 3-manifolds \cite{Fun13}, but also provide indispensable tools in the study of modular fusion categories, especially in  classification problems \cite{Rank5, NWZ-trans}. 

One would naturally extend to investigate the properties of mapping class group representations of higher genus surfaces and their applications. The analogue of $\SL(2,\bZ)$ in higher genus is the symplectic group $\Sp(2g, \bZ)$, and the natural surjection $\MCG{g}\twoheadrightarrow\Sp(2g, \bZ)$ has a nontrivial kernel $\torelli{g}$, called the Torelli group \cite{FM12}. Apparently, the associated TFT representation of $\MCG{g}$ factors through $\Sp(2g, \bZ)$ if and only if it vanishes on $\torelli{g}$. 

In this paper, we study modular fusion categories whose mapping class group representations vanish on $\torelli{g}$, or $g$-Torelli trivial modular fusion categories. Our main results, presented in Section \ref{s3}, are two classification results obtained from the triviality of the Torelli group actions. Firstly, we show in Theorem \ref{thm:triv-torelli-all-genus} that a modular fusion category is $g$-trivial for \textbf{any} $g \ge 3$ if and only if it is pointed. Our result is a refinement of the result in \cite[Prop.~6.7]{MW25}, where the author gave necessary and sufficient conditions for modular tensor categories to have trivial Torelli group actions for \textbf{all} genus $g$. More interestingly, motivated by the study of 2-Torelli trivial modular fusion categories, we study modular fusion categories whose adjoint subcategories are isotropic (having trivial twist). We obtain a full classification of such categories in Theorem \ref{t:iso_adj}, showing that a modular fusion category has isotropic adjoint subcategory if and only if it is a modular subcategory of the Drinfeld center of a pointed fusion category whose group of simple objects is abelian, extending the results in \cite{FF20}. We also obtained the congruence property of the $\Sp(2g, \bZ)$ representation of a $g$-Torelli trivial modular fusion categories in Theorem \ref{thm:congruence-kernel}. In Section 4, we prove in Theorem \ref{t:Bn3} that the groups $G$ satisfying $\zdr(\Vec_G)$ is 2-Torelli trivial but $\zdr(\Vec_G)_{\ad}$ is not isotropic are precisely the 2-Engel groups of nilpotent class 3. 

\subsection*{Acknowledgments}
S.N. was partially supported by the Simons Foundation MPS-TSM-00008039. Y.W.~was supported by NSFC grant No.\ 12301045 and No.\ 12571041, and he would like to thank Shuang Ming and Yuze Ruan for helpful discussions.

%========================================
\section{Preliminaries}
In this section, we set up notations and recall basic concepts on modular fusion categories and their projective mapping class groups representations.
%========================================
\subsection{Modular fusion categories}
\label{subsec:mtc} 
A \emph{fusion category} over $\bC$ is a semisimple finite tensor category over $\bC$ \cite{ENO05}. We denote a complete set of representatives of isomorphism classes of simple objects of a fusion category $\CC$ by $\irr(\CC)$. In particular, the tensor unit $\1$ is simple. The tensor product of $\cC$ endows $K_0(\CC)$, the Grothendieck group of $\CC$, with a ring structure that is characterized by the \emph{fusion rules}: For $a, b\in\irr(\CC)$,
\begin{equation}\label{eq:fusion-rules}
a\otimes b = \sum_{c\in\irr(\CC)} N_{a, b}^{c} c\,, \text{ with } N_{a, b}^{c} := \dim_{\BC}\CC(a\otimes b, c)\,.
\end{equation}
% We call $\{N_{a, b}^{c}\mid a, b, c \in \irr(\cC)\}$ the set of \emph{fusion coefficients} of $\cC$. 
For any $a \in \irr(\cC)$, the largest real eigenvalue of the fusion matrix $(N_a)_{b, c} := N_{a,b}^{c}$ of $a$ is called the Frobenius-Perron dimension of $a$, denoted by $\FP(a)$. The Frobenius-Perron dimension of $\cC$ is defined to be 
\begin{equation}
    \FP(\cC) := \sum_{a \in \irr(\cC)} \FP(a)^2\,.
\end{equation}

A fusion category $\CC$ is rigid,  i.e., every object $x \in \cC$ admits a left dual and a right dual. A left dual of $x$ is a triple $(x^*, \coev_x, \ev_x)$ with $x^*\in \CC$ such that the coevaluation $\coev_{x}: \1 \to x \otimes x^{*}$ and the evaluation $\ev_{x}: x^* \otimes x \to \1$ are morphisms satisfying the left rigidity conditions; and a right dual of $x$ is defined similarly (see \cite[Sec.~2.10]{EGNO}). If we fix a choice of a left dual for each object in $\cC$, then the left duality extends to a tenor functor $(-)^* : \CC \to \CC^{\op}$ and the double dual functor $(-)^{**}: \cC \to \cC$ is a tensor equivalence. An isomorphism $\piv: \id_\CC \xrightarrow{\sim} (-)^{**}$ of tensor functors is called a \emph{pivotal structure} on $\cC$, and the pair $(\cC, \piv)$ is called a pivotal fusion category. 

In a pivotal fusion category $(\cC, \piv)$, the (left) pivotal trace of an endomorphism $f: x \to x$ in $\cC$ is defined to be $\tr^{\piv}(f) := \ev_{x^{*}} \circ ((\piv_x\circ f) \otimes \id_{x^{*}}) \circ \coev_x \in \BC$, and the (left) pivotal dimension of $x$ is $d^{\piv}_x := \tr^{\piv}(\id_x)$. We call $\piv$ (or $\cC$) \emph{spherical} if $\tr^{\piv}(f) = \tr^{\piv}(f^*)$ for all endomorphism $f$ of $\CC$. For brevity,  we will omit $\piv$ when there is  no ambiguity. 
In \cite{Mug031}, M\"uger introduced the notion of the global dimension $\dim(\cC)$ of a fusion category $\cC$.
% defined for any fusion category $\cC$ the global dimension $\dim(\cC)$ which is independent of pivotal structure. However, w
When $\cC$ is spherical, we have
\begin{equation}
\dim(\CC) = \sum_{a \in \irr(\CC)} d_a^2\,.
\end{equation}
By \cite[Rmk.\ 2.5]{ENO05}, for any fusion category $\cC$, $\dim(\CC)$ is a totally positive cyclotomic integer. We denote the positive square root of $\dim(\CC)$ by $\sqrt{\dim(\CC)}$. 

Let $\CC$ be a braided fusion category with braiding $\beta_{x,y}: x \otimes y \xrightarrow{\sim} y \otimes x$ for $x, y \in \cC$, and let $\DD \subset \CC$ be a collection of objects in $\CC$. The \emph{M\"uger centralizer} of $\DD$ in $\CC$, denoted by $\zmug(\cD, \cC)$, is the full subcategory of $\CC$ containing objects $x \in \CC$ such that $\beta_{y, x} \circ \beta_{x, y} = \id_{x \otimes y}$ for all $y \in \DD$ \cite{Mug032}. In particular, the fusion subcategory $\zmug(\cC) := \zmug(\CC, \cC)$ is called the \emph{M\"uger center} of $\CC$. A braided fusion category $(\CC, \beta)$ is \emph{symmetric} if $\zmug(\cC) = \CC$; it is called \emph{non-degenerate} if $\zmug(\cC) \cong \Vec$, the category of finite-dimensional vector spaces over $\BC$. For example, the Drinfeld center $\zdr(\cC)$ of a fusion category $\cC$ \cite{Maj91,JS91} is a non-degenerate braided fusion category \cite[Cor.\ 3.9]{DGNO-BFC}.

A \emph{ribbon structure} on a braided fusion category $\cC$ with braiding $\be$ is a natural isomorphism $\theta: \id_{\cC} \xrightarrow{\sim} \id_{\cC}$ satisfying the twist equation 
\begin{equation}\label{eq:tw-def}
\theta_{x\otimes y} = (\theta_{x} \otimes \theta_y)\circ \beta_{y, x} \circ \beta_{x, y} 
\end{equation}
and $\theta_{x^*} = (\theta_x)^*$ for all $x, y \in \cC$. As is shown in \cite[Prop.\ 8.10.12]{EGNO} and \cite[p.38]{NS07}, a ribbon fusion category $\cC$ is a spherical braided fusion category, and vice versa. For $a \in \Irr(\cC)$, $\theta_a$ is a non-zero scalar multiple of $\id_a$. We will use the same notation to denote both $\theta_a$ and associated scalar for any simple object $a$. The $T$-matrix of $\CC$ is the diagonal matrix
\[
T_{a, b} := \delta_{a, b} \theta_{a},\ a, b \in \irr(\CC)\,.
\]
The (unnormalized) $S$-matrix of $\CC$ is defined to be
\[
S_{a, b} := \tr(\beta_{b, a^*}\circ \beta_{a^*, b}),
\  a, b \in \irr(\CC)\,.
\]
In particular, $S_{a, \1} = S_{\1, a} = d_a$ for $a \in \irr(\cC)$. 

A \emph{modular fusion category} (MFC) is a ribbon fusion category $\cC$ whose underlying braided fusion category is non-degenerate, or equivalently, its S-matrix is invertible \cite{Mug032}. The pair ($S$, 
$T$) is called the (unnormalized) \emph{modular data} of $\CC$. It is well-known that for a modular fusion category $\cC$, the assignment $\smat \mapsto S$, $\tmat \mapsto T$ defines a projective representation of $\SL_2(\bZ)$ \cite{Tur10}. 

We finish our review on modular fusion categories by setting up the conventions for the graphical calculus, which enables us to explicitly describe the projective mapping class group representations in the next section.

Let $\cC$ be a modular fusion category. By \cite[Thm.\ 2.2]{NS07-piv}, we can assume without loss of generality that $\cC$ is strictly pivotal, namely, the associativity constraints, the monoidal structure of the dual functor, the double dual functor and the pivotal structure of $\cC$ are all identities. All the graphs below are $\cC$-colored ribbon graphs in the sense of \cite[Ch.\ I.2]{Tur10}, which we read from top to bottom. Moreover, the tensor unit $\1 \in \cC$ is omitted in the graphs. For example, for the left dual $(x^*, \coev_x, \ev_x)$ of an object $x \in \cC$, the morphisms $\coev_x$ and $\ev_x$ are depicted as
\[\DualDef.\]
The right dual of $x$ can be defined via the strict the pivotal structure and the left dual of $x^*$, and we use cups and caps with arrows pointing to the right to denote the evaluation $\ev'_x := \ev_{x^*}$ and the coevaluation $\coev'_x:=\coev_{x^*}$ for the right dual of $x$. 

For any $x, y \in \cC$, we depict the braiding $\be_{x, y}$ and its inverse as
\[\be_{x, y} = \Brd\,, \quad \be_{x, y}^{-1} = \BrdInv\,.\]
% In view of \cite[Prop.\ 8.10.12]{EGNO}, we can already give a graphical presentation of the ribbon structure $\theta_x$ of $x \in\cC$ from the above graphical notations for duals and braidings in $\cC$ (see for example \cite[Thm.\ I.2.5]{Tur10}, where diagrams are read from bottom to top). However, for brevity, we choose to draw
We draw
\[\theta_x = \vcenter{\hbox{\scalebox{0.6}{
\begin{tikzpicture}[line width=2pt]
\node[draw, circle] (a) at (0,0) {$+1$}; 
\draw (a.north)--(0,1.5);
\draw (a.south)--(0,-1.5);
\node[right] at (0, 1.4) {\Large $x$};
\end{tikzpicture}}}}\,,
\quad\text{and}\quad
\theta^{-1}_x = \vcenter{\hbox{\scalebox{0.6}{
\begin{tikzpicture}[line width=2pt]
\node[draw, circle] (a) at (0,0) {$-1$}; 
\draw (a.north)--(0,1.5);
\draw (a.south)--(0,-1.5);
\node[right] at (0, 1.4) {\Large $x$};
\end{tikzpicture}}}}\]
for the ribbon structure.

For any triple of simple objects $(a, b, c) \in \irr(\CC)^3$, we fix a choice of basis 
\begin{equation} 
\{\hmu_j(a|b, c)\in \CC(a, b\ot c) \mid j = 1, ..., N_{b, c}^{a}\}
\end{equation} 
for the Hom-space $\CC(a, b \ot c)$, and
 use the following trivalent vertex to represent $\hmu_j(a|b, c)$:
\begin{equation}\label{eq:mu-hat} 
\hmu_j(a|b, c)=\TrivalentVertexHat\,.
\end{equation}

\begin{remark}\label{r:zero-tri-vert}
For completeness, for all triples $(a, b, c) \in \irr(\cC)^3$ and any $j \in \bZ$, we set $\hmu_j(a|b, c) = 0$ if $N_{b, c}^{a} = 0$, or if $N_{b, c}^{a} > 0$ and $j \notin [1, N_{b, c}^{a}]$.    
\end{remark}

When the Hom-space is 1-dimensional, we omit the index $j$. For example, for any $a \in \Irr(\cC)$, $\coev_a$ is a non-zero multiple of $\hmu_j(\1|a, a^*)$ in $\cC(\1, a \ot a^*)$, i.e., there exists $u_a \in \bC^\times$ such that 
\[\coev_a = \vcenter{\hbox{\scalebox{0.6}{
\begin{tikzpicture}[line width=2pt]
\draw[->-=.5] (0,0) arc(0:180:1);
\node[below] at (0,0) {\Large $a^*$};
\node[below] at (-2,-0.13) {\Large $a$};
\end{tikzpicture}}}} 
= u_a \cdot \hmu(\1|a, a^*)
= u_a \cdot \vcenter{\hbox{\scalebox{0.6}{
\begin{tikzpicture}[line width=2pt]
\draw (0,0) -- ++(225:1.5)
      (0,0) -- ++(315:1.5);
\draw ([shift=(225:0.3cm)]0,0) arc (225:315:0.3cm);
\node[below] at (-1.2, -1.13) {\Large $a$};
\node[below] at (1.2, -1) {\Large $a^*$};
\fill[black] (0,0) circle (3pt);
\end{tikzpicture}}}}\]

Since $\CC$ is semisimple, there exists a basis $\{\cmu_k(b, c|a) \in \cC(b \ot c, a)\mid k = 1, ..., N_{b, c}^a\}$ for $\CC(a, b \otimes c)$ such that 
\begin{equation}\label{eq:basis1}
\cmu_k(b, c|a) \circ \hmu_j(a|b, c) = \delta_{j,k} \id_{c} 
\quad\text{and}\quad 
\sum_{c \in \irr(\CC)} \sum_{j = 1}^{N_{a, b}^{c}} \hmu_j(a|b, c) \circ \cmu_j(b, c|a) = \id_{a \o b}\,,
\end{equation}
and we denote
\[\cmu_j(a, b|c)=\TrivalentVertexCheck\,.\]
For brevity, we adopt the following convention: If an index in a summation is used to label objects in $\CC$ in a graph, we take the sum over $\Irr(\CC)$; if an index is used to label the basis vectors in a Hom-space, we sum the index from 1 to the dimension of the corresponding Hom-space. For example, \eqref{eq:basis1} in our graphical calculus reads
\begin{equation}\label{eq:basis2}
\StraightLine{b} \quad \StraightLine{c}
=\sum_{a, j} \CupCap
\quad\text{and}\quad
\CapCup=\delta_{j, k} \StraightLine{a}\,.
\end{equation}
%=====================================
\subsection{Mapping class group representations from modular fusion categories}\label{sec:MCG-rep}
%========================================
Let $g \ge 0$ be an integer, and $\Sigma_g$ be a closed (compact with empty boundary) oriented smooth surface of genus $g$. The \emph{mapping class group} of $\Sigma_g$, denoted by $\MCG{g}$, is the group of isotopy classes of orientation-preserving self-homeomorphisms of $\Sigma_g$, i.e., $\MCG{g} = \operatorname{Homeo}^+(\Sigma_g)/\isotopy$. Elements in $\MCG{g}$ are called mapping classes. Note that when $g=0$, $\MCG{0}$ is trivial \cite[Sec.\ 2.2]{FM12}, so from now on, we assume $g \ge 1$.

It is well-known that $\MCG{g}$ is generated by a finite set of \emph{Dehn twists} \cite{Deh38} (see \cite[Chap.\ 3-4]{FM12} for details). Let $\gamma \subset \Sigma_g$ be a simple closed curve, i.e., the image of an embedding $\bS^1\hookrightarrow \Sigma_g$. The Dehn twist along $\gamma$, denoted by $\dt_\gamma$, is locally a twist map $(\om, t) \mapsto (\om+2\pi t, t)$ on a regular neighborhood $N \cong \bR/2\pi \bZ \times [0, 1]$ of $\gamma$ (see \eqref{eq:Dehn-Twist}) in $\Sigma_g$, and it is the identity outside of $N$. 
\begin{equation}\label{eq:Dehn-Twist}
\vcenter{\hbox{\scalebox{0.6}{
\begin{tikzpicture}[line width=2pt]
\begin{scope}[xshift=-12cm, scale=0.5]
\draw (0,0) ellipse (1 and 3);
\draw (8,3) arc (90 : -90 : 1 and 3);
\draw [dashed] (8,3) arc (90 : 270 : 1 and 3);
\draw (0,3) to (8,3);
\draw (0,-3) to (8,-3);
\draw [purple] (4,3) arc (90 : -90 : 1 and 3);
\draw [dashed, purple] (4,3) arc (90 : 270 : 1 and 3);
\draw [blue] (1,0) to (9,0);
\node (g) at (5.4,1.6) {\Large $\gamma$};
\end{scope}
\begin{scope}[xshift=-5cm]
\node[above] (k) at (0,0) {\Large $\dt_\gamma$};
\draw [->] (-1.2,0) to (1.2,0);
\end{scope}
\begin{scope}[xshift=-2cm, scale=0.5]
\draw (0,0) ellipse (1 and 3);
\draw (8,3) arc (90 : -90 : 1 and 3);
\draw [dashed] (8,3) arc (90 : 270 : 1 and 3);
\draw (0,3) to (8,3);
\draw (0,-3) to (8,-3);
\draw [blue] (1,0) to [out=0, in=-60] (4,3);
\draw [dashed, blue] (4,3) arc (90 : 270 : 1 and 3);
\draw [blue] (9,0) to [out=180, in=60] (4,-3);
\end{scope}
\end{tikzpicture}}}}
\end{equation}
In \cite{Lic64}, Lickorish proved that the Dehn twists along the curves $m$'s, $\ell$'s and $e$'s shown in Figure \ref{fig:lic-gen} form a generating set of $\MCG{g}$. 
\begin{figure}[ht]
\centering
$\vcenter{\hbox{\scalebox{0.6}{
\begin{tikzpicture}[line width=2pt]
\draw (0,2) arc (90:270:2);
\draw (10, -2) arc (-90:90:2);
\draw (0,2)--(10,2); \draw (0,-2)--(10,-2);
\HoleWithCurve{(0,0)} \HoleWithCurve{(4,0)}
\LeftEquator{(0,0)} \LeftEquator{(4,0)}
\RightEquator{(4,0)} \RightEquator{(10,0)}
\node at (7,0) {\Large $\cdots$};
\HoleWithCurve{(10,0)}
\node at (0.7,1.3) {\Large $m_1$};
\node at (0,-1) {\Large $\ell_1$};
\node at (2,0.85) {\Large $e_1$};
\node at (8.2,0.85) {\Large $e_{g-1}$};
\begin{scope}[shift={(4,0)}]
\node at (0.7,1.3) {\Large $m_2$};
\node at (0,-1) {\Large $\ell_2$};
\node at (2,0.85) {\Large $e_2$};
\end{scope}
\begin{scope}[shift={(10,0)}]
\node at (0.7,1.3) {\Large $m_g$};
\node at (0,-1) {\Large $\ell_g$};
\end{scope}
\end{tikzpicture}}}}$
\caption{Lickorish's generating set of $\MCG{g}$.}
\label{fig:lic-gen}
\end{figure}
To give a simple description of the mapping class group representation associated to a modular fusion category, we choose the following alternative set of generators of $\MCG{g}$:
\begin{equation}\label{eq:mcg-gen}
\{\dt_{m_{1}}, ..., \dt_{m_{g}}, \dt_{e_{1}}, ..., \dt_{e_{g-1}}, \cmcs_{1}, ..., \cmcs_{g}\}\,,
\end{equation}
where $\cmcs_{k} := \dt_{m_{k}} \circ \dt_{\ell_{k}}\circ \dt_{m_{k}}$ for $1 \le k \le g$.

Let $\cC$ be a modular fusion category, and consider the Hom-space
\begin{equation}\label{eq:RT-space}
\RTspace{\cC}{g} := \CC(\1, I^{\ot g}) \quad\text{where}\quad
I = \bigoplus_{a \in \irr(\CC)} a \ot a^*\,.
\end{equation}
By the semisimplicity of $\cC$, we have a decomposition
\begin{equation}\label{eq:W-decomp}
\RTspace{\cC}{g} \cong 
\bigoplus_{\substack{a_1, ..., a_g\\b_1, ..., b_{2g-2}}}
\CC(b_1, a_1 \ot a_1^*) \ot_{\bC} \cC(b_2, b_1 \ot a_2) \ot_{\bC} \cdots \ot_{\bC} \CC(b_{2g-2}, b_{2g-3} \ot a_g)\,.
\end{equation}
In view of this decomposition, we choose a basis for $\RTspace{\cC}{g}$ which consists of ``left branching trees'' of the trivalent vertices $\hmu$ in \eqref{eq:mu-hat} as follows. 

For any tuples $\bm{a} = (a_1, ..., a_g) \in \irr(\CC)^g$, $\bm{b} = (b_1, ..., b_{2g-2}) \in \irr(\cC)^{2g-2}$ and $\bm{j} = (j_1, ..., j_{2g-2}) \in \bZ^{2g-2}$, define
\begin{equation}\label{eq:W-basis}
\begin{split}
& \Wbase{g}{a}{b}{j} := 
\vcenter{\hbox{\scalebox{0.6}{
\begin{tikzpicture}[line width=2pt]
\foreach \x in {0,1,2,5}{\BranchDown{(\x, \x)}}
\draw (2,2)--(2.5,2.5);
\begin{scope}[shift={(4,4)}]
\draw (0,0)--(-0.5,-0.5); \draw (0,0)--(1,-1);
\draw ([shift=(225:0.3cm)]0,0) arc (225:315:0.3cm);
\fill[black] (0,0) circle (3pt);
\end{scope}
\node[below] at (-1,-1) {\Large $a_1^{\phantom{*}}$}; \node[below] at (1,-1) {\Large $a_1^*$};
\draw (2,0)--(3,-1); \draw (3,1)--(5,-1);\node[below] at (3,-1) {\Large $a_2^{\phantom{*}}$}; \node[below] at (5,-1) {\Large $a_2^*$};
\draw (5,3)--(9,-1); \draw (6,4)--(11,-1);\node[below] at (9,-1) {\Large $a_g^{\phantom{*}}$}; \node[below] at (11,-1) {\Large $a_g^*$};
\node[below] at (0,-0.3) {\Large $j_1$}; \node at (0.3,0.8) {\Large $b_1$};
\node[below] at (1,0.7) {\Large $j_2$};\node at (1.3,1.8) {\Large $b_2$};
\node[below] at (2,1.7) {\Large $j_3$}; \node at (4,4.8) {\Large $b_{2g-2}$};
\node[below] at (4.1,3.5) {\Large $j_{2g-2}$};
\foreach \x in {1,2,3} {\fill[black] (2.4+0.3*\x,2.4+0.3*\x) circle (1pt);} 
\end{tikzpicture}}}}  
\end{split}
\end{equation}
Algebraically, we have (cf.\ Remark \ref{r:zero-tri-vert})
\begin{equation} 
\label{eq:W-basis-alg}
\begin{split} 
\Wbase{g}{a}{b}{j}
=& \left(\hmu_{j_{1}}(b_1|a_1, a_1^*) \ot \id^{\ot (2g-2)} \right)
\circ\left(\hmu_{j_{2}}(b_2|b_1, a_2) \ot \id^{\ot (2g-3)}\right)\\
&\hspace*{5em}\circ\cdots\circ\left(\hmu_{j_{2g-2}}(b_{2g-2}| b_{2g-3}, a_g) \ot \id \right) \circ \hmu(\1|b_{2g-2}, a_{g}^*)\,.  
\end{split}
\end{equation}
Note that $\Wbase{g}{a}{b}{j} \ne 0$ if and only if every trivalent vertex in \eqref{eq:W-basis} is non-zero. 

\begin{definition}\label{def:W-base}
Let $\cC$ be a modular fusion category. For any integer $g \ge 1$, an element $(\bm{a}, \bm{b}, \bm{j})\in 
\irr(\cC)^{g}  \times \irr(\cC)^{2g-2} \times \bZ^{2g-2}$
is called \emph{admissible} if  $\Wbase{g}{a}{b}{j} \ne 0$. We set 
\[W_{\cC, g} := \{\Wbase{g}{a}{b}{j} \mid (\bm{a}, \bm{b}, \bm{j}) \text{ is admissible}\}\,.\qedhere\] 
\end{definition}
We immediately have the following lemma.
\begin{lemma}\label{lem:W-basis}
For any $g \ge 1$, $W_{\cC, g}$ is a basis for $\RTspace{\cC}{g}$. \qed
\end{lemma}
To save space, we introduce the following ribbon diagrams as alternative graphical presentations of $\Wbase{g}{a}{b}{j}$ that will be used later:
\begin{equation}\label{eq:RTbasis}
\vcenter{\hbox{\scalebox{0.6}{
\begin{tikzpicture}[line width=2pt]
\draw (-3.5,-0.8) rectangle (3.5,0.8); 
\node at (0,0) {\LARGE $\Wbase{g}{a}{b}{j}$};
%========================================
\foreach\x in {1,2}{
\draw(2*\x-5,-0.8)--(2*\x-5,-2.5); \node[below]at(2*\x-5,-2.5){\Large $a_{\x}^{\phantom{*}}$};
\draw(2*\x-4,-0.8)--(2*\x-4,-2.5);\node[below]at(2*\x-4,-2.5){\Large $a_{\x}^*$};}
%========================================
\draw (2,-0.8)--(2,-2.5);\node[below] at (2,-2.5) {\Large $a_g^{\phantom{*}}$};
\draw (3,-0.8)--(3,-2.5);\node[below] at (3,-2.5) {\Large $a_g^*$};
\foreach\x in {1,2,3} {\fill[black] (0.6+0.2*\x,-1.8) circle (1pt);}
\end{tikzpicture}}}}
:=
\vcenter{\hbox{\scalebox{0.6}{
\begin{tikzpicture}[line width=2pt]
\draw (0,2)--(5.2,2); \draw(6.8,2)--(9,2);
\node[below] at (0,0) {\Large $a_1^{\phantom{*}}$};
\node[below] at (1.5,0) {\Large $a_1^{*}$};
\node[below] at (3,0) {\Large $a_2^{\phantom{*}}$};
\node[below] at (4.5,0) {\Large $a_2^{{*}}$};
\node[below] at (7.5,0) {\Large $a_g^{\phantom{*}}$};
\node[below] at (9,0) {\Large $a_g^{{*}}$};
%========================================
\foreach \x in {0,1,2,3,5,6}{\draw (1.5*\x, 2)--(1.5*\x,0);}
%========================================
\foreach\x in {1,2,3,5,6}
{\begin{scope}[shift={(1.5*\x, 2)}]
\draw ([shift=(180:0.3cm)]0,0) arc (180:270:0.3cm);
\fill[black] (0,0) circle (3pt);
\end{scope}}
%========================================
\foreach\x in {1,2,3}{\fill[black] (5.6+0.2*\x,2) circle (1pt);}
%========================================
\foreach\x in {1,2,3}{\node[left] at (1.5*\x-0.05,1.6) {\Large $j_{\x}$};}
%========================================
\node[left]at(7.5-0.05,1.6){\Large $j_{2g-2}$};
\node[above] at (2.25,2) {\Large $b_1$};
\node[above] at (2.25+1.5,2) {\Large $b_2$};
\node[above] at (8.25,2) {\Large $b_{2g-2}$};
\end{tikzpicture}}}}\,.
\end{equation}

Now we are ready to describe the projective representation of $\MCG{g}$ on $\RTspace{\cC}{g}$. 
This projective representation is afforded by the Reshetikhin-Turaev topological quantum field theory (RT-TQFT) associated to $\cC$ \cite[Ch.\ IV]{Tur10}, denoted by $\RT_{\cC}$, which is a symmetric monoidal \emph{projective} functor from the (2+1)D cobordism category to $\Vec$. In particular, $\RT_{\cC}$ assigns $\Sigma_g$ to the vector space $\RTspace{\cC}{g}$.

For any (representative of) $f \in \mcg{g}$, its mapping cylinder can be viewed as a cobordism from $\Sigma_g$ to itself, whose cobordism class, denoted by $\cyl(f)$, depends only on the mapping class $f$. 
Hence, $\RT_{\cC}$ assigns to $f$ a linear operator
\[\RTmat{\cC}{g}(f) := \RT_{\cC}(\cyl(f)): \RTspace{\cC}{g} \to \RTspace{\cC}{g}\,.\]
With respect to our fixed choice of basis $W_{\cC,g}$ for $\RTspace{\cC}{g}$, the matrix coefficients of $\RTmat{\cC}{g}(f)$ can be computed explicitly according to the recipe detailed in \cite[Ch.\ IV]{Tur10}. Fix an ordering on $W_{\cC, g}$, we can identify $\Wbase{g}{a}{b}{j} \in W_{\cC, g}$ with a column vector of length
\begin{equation}
\label{eq:dim-Vg}
\ncg{\cC}{g} := \dim_{\bC}(\RTspace{\cC}{g}) = |W_{\cC,g}|\,.
\end{equation} 
\begin{definition} For $M \in \GL(n, \bC)$, we denote by $[M]$ its image of $M$ under the natural surjection  $\GL(n, \bC) \to \PGL(n, \bC)$. 
\end{definition}
In this way, the linear operator $\RTmat{\cC}{g}(f)$ is represented by a square matrix of size $\ncg{\cC}{g}$, whose column corresponding to $\Wbase{g}{a}{b}{j}$ is 
\begin{equation}
\RTmat{\cC}{g}(f)(\Wbase{g}{a}{b}{j})\,,
\end{equation} 
and we simply write $\RTmat{\cC}{g}(f) \in \GL(\ncg{\cC}{g}, \bC)$.  Then the assignment 
\begin{equation}\label{eq:RT-rep}
\RTrep{\cC}{g}: \MCG{g} \to \PGL(\ncg{\cC}{g}, \bC)\,,\quad f \mapsto [\RTmat{\cC}{g}(f)]	
\end{equation}
defines a group homomorphism, i.e., $\RTrep{\cC}{g}$ is a projective representation of $\MCG{g}$ \cite[Ch.\ IV]{Tur10}.

We end this section by giving an explicit description of $\RTrep{\cC}{g}(\dt_\gamma)$ when $\dt_\gamma$ is a Dehn twist of  a simple closed curve $\gamma$ in $\Sigma_g$. Let $\hmu_g(\bm{a}, \bm{b}, \bm{j}) \in W_{\cC,g}$. Define
\begin{equation}\label{eq:rib-graph}
\hmu'_{g}(\bm{a}, \bm{b}, \bm{j}) := 
\vcenter{\hbox{\scalebox{0.6}{
\begin{tikzpicture}[line width=2pt]
\begin{scope}[shift={(-1,-1)}]
\draw (0,2)--(7,2); \draw(9,2)--(12,2);
\node[left] at (0.1,0) {\Large $a_1^{\phantom{*}}$};
\node[right] at (1.9,0) {\Large $a_1^{*}$};
\node[left] at (4.1,0) {\Large $a_2^{\phantom{*}}$};
\node[right] at (5.9,0) {\Large $a_2^{{*}}$};
\node[left] at (10.1,0) {\Large $a_g^{\phantom{*}}$};
\node[right] at (11.9,0) {\Large $a_g^{{*}}$};
%========================================
\foreach \x in {0,1,2,3,5,6}{\draw (2*\x, 2)--(2*\x,0.5);}
%========================================
\foreach\x in {1,2,3,5,6}
{\begin{scope}[shift={(2*\x, 2)}]
\draw ([shift=(180:0.3cm)]0,0) arc (180:270:0.3cm);
\fill[black] (0,0) circle (3pt);
\end{scope}}
%========================================
\foreach\x in {1,2,3}{\node[left] at (2*\x-0.05,1.6) {\Large $j_{\x}$};}
%========================================
\node[left]at(10,1.6){\Large $j_{2g-2}$};
\node[above] at (3,2) {\Large $b_1$};
\node[above] at (5,2) {\Large $b_2$};
\node[above] at (11,2) {\Large $b_{2g-2}$};	
\end{scope}
%========================================
\foreach\x in {0,4,10}
{\begin{scope}[shift={(1+\x,-2.28)}, xscale=2, yscale=1.2]
\draw[red] (0,0) arc (0:180:0.5 and 1);  
\draw (0,1.5) arc (360:180:0.5 and 1);
\fill[white] (-0.18,0.73) circle (3pt);  
\fill[white] (-0.82,0.73) circle (3pt);  
\draw[red] (0,0) arc (0:90:0.5 and 1);  
\draw (-0.5,0.5) arc (270:180:0.5 and 1);
\draw(-0.45,0.5)--++(-0.13,0.18);\draw(-0.45,0.5)--++(-0.18,-0.13);
\end{scope}
\begin{scope}[shift={(1+\x,-2.28)}, yscale=-1.2, xscale=2
]
\draw[red] (0,0) arc (0:180:0.5 and 1);  
\draw[red] (0,1.5) arc (360:180:0.5 and 1);
\fill[white] (-0.18,0.73) circle (3pt);  
\fill[white] (-0.82,0.73) circle (3pt);  
\draw[red] (0,0) arc (0:90:0.5 and 1);  
\draw[red] (-0.5,0.5) arc (270:180:0.5 and 1);
\end{scope}}
\foreach\x in {1,2,3}{\fill[black] (6.5+0.2*\x,1) circle (1pt);}
\end{tikzpicture}}}}
\end{equation}
where the red color stands for the Kirby color 
\begin{equation}\label{eq:kirby}
\Omega := \sum_{y \in \irr(\cC)} d_y y \in K_0(\cC) \ot_{\bZ}\bC\,.
\end{equation}
By the cutting property of the Kirby color (see for example \cite[Cor.~3.1.11]{BK01}), there exists a nonzero constant $\alpha(\bm{a}) \in \bC^{\times}$ depending on $\bm{a}$  such that 
\begin{equation}
\hmu'_{g}(\bm{a}, \bm{b}, \bm{j}) = \alpha(\bm{a}) \cdot \hmu_{g}(\bm{a}, \bm{b}, \bm{j})\,.
\end{equation}
In particular, the codomains of $\hmu'_{g}(\bm{a}, \bm{b}, \bm{j})$ and $\hmu_{g}(\bm{a}, \bm{b}, \bm{j})$ agree. A tubular neighborhood of the top component of $\hmu'_{g}(\bm{a}, \bm{b}, \bm{j})$ is a handlebody of genus $g$ in $\bR^3$ with boundary surface $\Sigma_g$:
\begin{equation}
\vcenter{\hbox{\scalebox{0.6}{
\begin{tikzpicture}[line width=2pt]
\draw[gray] (-1,2) arc (90:270:2);
\draw[gray] (11, -2) arc (-90:90:2);
\draw[gray] (-1,2)--(11,2); \draw[gray] (-1,-2)--(11,-2);
\begin{scope}[shift={(0,-0.5)}]
\grayhole{0,0}; \grayhole{4,0};\grayhole{10,0};
\end{scope}
%========================================
% Embedding of basis
%========================================
\begin{scope}[shift={(-1,-1)}]
\draw (0,2)--(7,2); \draw(9,2)--(12,2);
\node[left] at (0.1,0.5) {\Large $a_1^{\phantom{*}}$};
\node[right] at (1.9,0.5) {\Large $a_1^{*}$};
\node[left] at (4.1,0.5) {\Large $a_2^{\phantom{*}}$};
\node[right] at (5.9,0.5) {\Large $a_2^{{*}}$};
\node[left] at (10.1,0.5) {\Large $a_g^{\phantom{*}}$};
\node[right] at (11.9,0.5) {\Large $a_g^{{*}}$};
%========================================
\foreach \x in {0,1,2,3,5,6}{\draw (2*\x, 2)--(2*\x,0.5);}
%========================================
\foreach\x in {1,2,3,5,6}
{\begin{scope}[shift={(2*\x, 2)}]
\draw ([shift=(180:0.3cm)]0,0) arc (180:270:0.3cm);
\fill[black] (0,0) circle (3pt);
\end{scope}}
%========================================
\foreach\x in {1,2,3}{\node[left] at (2*\x-0.05,1.6) {\Large $j_{\x}$};}
%========================================
\node[left]at(10,1.6){\Large $j_{2g-2}$};
\node[above] at (3,2) {\Large $b_1$};
\node[above] at (5,2) {\Large $b_2$};
\node[above] at (11,2) {\Large $b_{2g-2}$};	
\end{scope}
%========================================
\foreach\x in {0,4,10}
{\begin{scope}[shift={(1+\x,-2.28)}, xscale=2, yscale=1.2]
\draw[red] (0,0.2) arc (0:100:0.5 and 1.7);  
\draw[red, dashed] (-1,0.2) arc (180:100:0.5 and 1.7);  
\draw (0,1.5) arc (360:180:0.5 and 1);
\fill[white] (-0.1,1) circle (5pt);  
\fill[white] (0,0.24) circle (3pt);  
\fill[white] (-0.93,1) circle (5pt);
\draw[red] (0,0.2) arc (0:90:0.5 and 1.7);  
\draw (-0.5,0.5) arc (270:180:0.5 and 1);
\draw(-0.45,0.5)--++(-0.13,0.18);\draw(-0.45,0.5)--++(-0.18,-0.13);
\end{scope}
\begin{scope}[shift={(1+\x,-2)}, yscale=-1.2, xscale=2
]
\draw[red] (0,0) arc (0:180:0.5 and 1.7);  
\draw[red] (0,2.2) arc (360:180:0.5 and 1.7);
\fill[white] (-0.1,1.1) circle (5pt);  
\fill[white] (-0.92,1.1) circle (5pt);  
\draw[red] (0,0) arc (0:90:0.5 and 1.7);  
\draw[red] (-0.5,0.5) arc (270:180:0.5 and 1.7);
\fill[white] (-1,0) circle (3pt);    
\draw[gray] (-1.1,0)--(-0.8,0);
\end{scope}}
\foreach\x in {1,2,3}{\fill[black] (6.5+0.2*\x,1) circle (1pt);}
\draw[gray] (-1,2) arc (90:270:2);
\end{tikzpicture}}}}    
\end{equation}
Draw $\gamma$ on $\Sigma_g$, take a regular neighborhood $\fN_\gamma\subset \Sigma_g$ of $\gamma$, and view $\fN_\gamma$ as a ribbon graph in $\bR^3$. Endow $\fN_\gamma$ with a $(-1)$-framing and label it by the Kirby color $\Omega$ in \eqref{eq:kirby}. Then we obtain a $\cC$-colored ribbon graph $\hmu'_{g}(\bm{a}, \bm{b}, \bm{j}) \cup \fN_{\gamma}^{(-1)}(\Omega)$, where $\fN_{\gamma}^{(-1)}(\Omega)$ is linked with $\hmu'_{g}(\bm{a}, \bm{b}, \bm{j})$. Since $\fN_{\gamma}^{(-1)}(\Omega)$ has no free ends, using graphical calculus, we can identify the $\cC$-colored ribbon graph $\hmu'_{g}(\bm{a}, \bm{b}, \bm{j}) \cup \fN_{\gamma}^{(-1)}(\Omega)$ with a vector in $\cC(\1, I^{\o g})=\RTspace{\cC}{g}$. According to \cite[Ch.~IV]{Tur10}, there exists $\ld \in \bC^\times$ such that
\begin{equation}\label{eq:2-23}
\RTmat{\cC}{g}(\dt_\gamma)(\Wbase{g}{a}{b}{j}) = \ld\alpha(\bm{a})^{-1}\cdot  \hmu'_{g}(\bm{a}, \bm{b}, \bm{j}) \cup \fN_{\gamma}^{(-1)}(\Omega)
\end{equation}
for all $\Wbase{g}{a}{b}{j} \in W_{\cC,g}$, and both $\alpha$ and $\ld$ are independent of $(\bm{a}, \bm{b}, \bm{j})$.

We finish this subsection with a few examples that will be used in the next sections.

\begin{example}\label{ex:g=2}
Let $g\ge 2$ and let $\varpi$ be a separating curve bounding a subsurface of genus 1:
\[\vcenter{\hbox{\scalebox{0.6}{
\begin{tikzpicture}[line width=2pt]
\draw (0,2) arc (90:270:2);
\draw (9, -2) arc (-90:90:2);
\draw (0,2)--(9,2); \draw (0,-2)--(9,-2);
\hole{0,-0.5};\hole{3,-0.5};\hole{6,-0.5};
\node at (9,0) {\Large $\cdots$};
\begin{scope} 
\draw[blue, ->-=0.5] (1.5, -2) arc (-90:90:0.5 and 2);
\draw[blue, dashed] (1.5, -2) arc (270:90:0.5 and 2);
\node[right] at (2,-1) {\Large $\varpi$};
\end{scope}
\end{tikzpicture}}}}\]
We now compute $\RTmat{\cC}{g}(\dt_{\varpi})$ according to the algorithm described above. First, we put $\varpi$ on the boundary of the tubular neighborhood  of $\Wbase{g}{a}{b}{j}$ and obtain
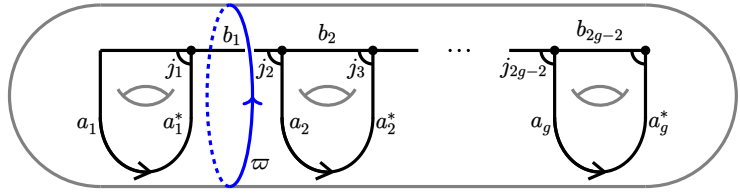
\[\vcenter{\hbox{\scalebox{0.6}{
\begin{tikzpicture}[line width=2pt]
\draw[gray] (-1,2) arc (90:270:2);
\draw[gray] (11, -2) arc (-90:90:2);
\draw[gray] (-1,2)--(11,2); \draw[gray] (-1,-2)--(11,-2);
\begin{scope}[shift={(0,-0.5)}]
\grayhole{0,0}; \grayhole{4,0};\grayhole{10,0};
\end{scope}
% \node at (7,0) {\Large $\cdots$};\node at (7,1) {\Large $\cdots$};
%========================================
% Embedding of basis
%========================================
\begin{scope}[shift={(-1,-1)}]
\draw (0,2)--(7,2); \draw(9,2)--(12,2);
\node[left] at (0.1,0.4) {\Large $a_1^{\phantom{*}}$};
\node[left] at (2,0.4) {\Large $a_1^{*}$};
\node[right] at (4,0.4) {\Large $a_2^{\phantom{*}}$};
\node[right] at (5.9,0.4) {\Large $a_2^{{*}}$};
\node[left] at (10.1,0.4) {\Large $a_g^{\phantom{*}}$};
\node[right] at (11.9,0.4) {\Large $a_g^{{*}}$};
%========================================
\foreach \x in {0,1,2,3,5,6}{\draw (2*\x, 2)--(2*\x,0.5);}
%========================================
\foreach\x in {1,2,3,5,6}
{\begin{scope}[shift={(2*\x, 2)}]
\draw ([shift=(180:0.3cm)]0,0) arc (180:270:0.3cm);
\fill[black] (0,0) circle (3pt);
\end{scope}}
%========================================
\foreach\x in {1,2,3}{\node[left] at (2*\x,1.6) {\Large $j_{\x}$};}
%========================================
\node[left]at(10,1.6){\Large $j_{2g-2}$};
\node[above] at (2.9,1.95) {\Large $b_1$};
\fill[white] (3.2,1.9) rectangle ++(0.185,0.22);
\node[above] at (5,2) {\Large $b_2$};
\node[above] at (11,2) {\Large $b_{2g-2}$};	
\end{scope}
%========================================
\foreach\x in {0,4,10}
{\begin{scope}[shift={(1+\x,-2.28)}, xscale=2, yscale=1.2]
\draw (0,1.5) arc (360:180:0.5 and 1);
\draw (-0.5,0.5) arc (270:180:0.5 and 1);
\draw(-0.45,0.5)--++(-0.13,0.18);\draw(-0.45,0.5)--++(-0.18,-0.13);
\end{scope}
\begin{scope}[shift={(1+\x,-2)}, yscale=-1.2, xscale=2
]
\end{scope}}
\foreach\x in {1,2,3}{\fill[black] (6.5+0.2*\x,1) circle (1pt);}
\begin{scope}[shift={(0.35,0)}] 
\draw[blue, ->-=0.5] (1.5, -2) arc (-90:90:0.5 and 2);
\draw[blue, dashed] (1.5, -2) arc (270:90:0.5 and 2);
\node[right] at (1.8,-1.5) {\Large $\varpi$};
\end{scope}
\end{tikzpicture}}}}\]
Next, push $\varpi$ inside the handlebody, color it by $\Omega$ and endow the ribbon graph with $(-1)$-framing. We have 
\[\hmu'_{g}(\bm{a}, \bm{b}, \bm{j})\cup \fN_{\varpi}^{(-1)}(\Omega)= 
\vcenter{\hbox{\scalebox{0.6}{
\begin{tikzpicture}[line width=2pt]
\begin{scope}[shift={(-1,-1)}]
\draw (0,2)--(7,2); \draw(9,2)--(12,2);
\node[left] at (0.1,0) {\Large $a_1^{\phantom{*}}$};
\node[right] at (1.9,0) {\Large $a_1^{*}$};
\node[left] at (4.1,0) {\Large $a_2^{\phantom{*}}$};
\node[right] at (5.9,0) {\Large $a_2^{{*}}$};
\node[left] at (10.1,0) {\Large $a_g^{\phantom{*}}$};
\node[right] at (11.9,0) {\Large $a_g^{{*}}$};
%========================================
\foreach \x in {0,1,2,3,5,6}{\draw (2*\x, 2)--(2*\x,0.5);}
%========================================
\foreach\x in {1,2,3,5,6}
{\begin{scope}[shift={(2*\x, 2)}]
\draw ([shift=(180:0.3cm)]0,0) arc (180:270:0.3cm);
\fill[black] (0,0) circle (3pt);
\end{scope}}
%========================================
\foreach\x in {1,2,3}{\node[left] at (2*\x+0.05,1.55) {\Large $j_{\x}$};}
%========================================
\node[left]at(10,1.6){\Large $j_{2g-2}$};
\node[above] at (3,1.9) {\Large $b_1$};
\begin{scope}
\draw[red] (3,1.5) arc (-90:-140:0.6 and 1);
\fill[white] (3.42,1.9) rectangle ++(0.185,0.22);
\draw[red] (3,1.5) arc (-90:200:0.6 and 1);  
\end{scope}
\begin{scope}[red, shift={(3.55,2.8)}]
\draw[fill=white] (0,0) circle (10pt);
\node at (0,0) {$-1$};
\end{scope}
\node[above] at (5,2) {\Large $b_2$};
\node[above] at (11,2) {\Large $b_{2g-2}$};	
\end{scope}
%========================================
\foreach\x in {0,4,10}
{\begin{scope}[shift={(1+\x,-2.28)}, xscale=2, yscale=1.2]
\draw[red] (0,0) arc (0:180:0.5 and 1);  
\draw (0,1.5) arc (360:180:0.5 and 1);
\fill[white] (-0.18,0.73) circle (3pt);  
\fill[white] (-0.82,0.73) circle (3pt);  
\draw[red] (0,0) arc (0:90:0.5 and 1);  
\draw (-0.5,0.5) arc (270:180:0.5 and 1);
\draw(-0.45,0.5)--++(-0.13,0.18);\draw(-0.45,0.5)--++(-0.18,-0.13);
\end{scope}
\begin{scope}[shift={(1+\x,-2.28)}, yscale=-1.2, xscale=2
]
\draw[red] (0,0) arc (0:180:0.5 and 1);  
\draw[red] (0,1.5) arc (360:180:0.5 and 1);
\fill[white] (-0.18,0.73) circle (3pt);  
\fill[white] (-0.82,0.73) circle (3pt);  
\draw[red] (0,0) arc (0:90:0.5 and 1);  
\draw[red] (-0.5,0.5) arc (270:180:0.5 and 1);
\end{scope}}
\foreach\x in {1,2,3}{\fill[black] (6.5+0.2*\x,1) circle (1pt);}
\end{tikzpicture}}}}\]
By \cite[Eq.\ (3.1.6)]{BK01}, there exists $\ld' \in \bC^\times$,  independent of $(\bm{a}, \bm{b}, \bm{j})$, such that 
\[\begin{split}
&\hmu'_{g}(\bm{a}, \bm{b}, \bm{j})\cup \fN_{\varpi}^{(-1)}(\Omega)= 
\ld' \cdot 
\vcenter{\hbox{\scalebox{0.6}{
\begin{tikzpicture}[line width=2pt]
\begin{scope}[shift={(-1,-1)}]
\draw (0,2)--(7,2); \draw(9,2)--(12,2);
\node[left] at (0.1,0) {\Large $a_1^{\phantom{*}}$};
\node[right] at (1.9,0) {\Large $a_1^{*}$};
\node[left] at (4.1,0) {\Large $a_2^{\phantom{*}}$};
\node[right] at (5.9,0) {\Large $a_2^{{*}}$};
\node[left] at (10.1,0) {\Large $a_g^{\phantom{*}}$};
\node[right] at (11.9,0) {\Large $a_g^{{*}}$};
%========================================
\foreach \x in {0,1,2,3,5,6}{\draw (2*\x, 2)--(2*\x,0.5);}
%========================================
\foreach\x in {1,2,3,5,6}
{\begin{scope}[shift={(2*\x, 2)}]
\draw ([shift=(180:0.3cm)]0,0) arc (180:270:0.3cm);
\fill[black] (0,0) circle (3pt);
\end{scope}}
%========================================
\foreach\x in {1,2,3}{\node[left] at (2*\x+0.05,1.55) {\Large $j_{\x}$};}
%========================================
\node[left]at(10,1.6){\Large $j_{2g-2}$};
\node[above] at (2.4,1.9) {\Large $b_1$};
\begin{scope}[shift={(3,2)}]
\draw[fill=white] (0,0) circle (10pt);
\node at (0,0) {$+1$};
\end{scope}
\node[above] at (5,2) {\Large $b_2$};
\node[above] at (11,2) {\Large $b_{2g-2}$};	
\end{scope}
%========================================
\foreach\x in {0,4,10}
{\begin{scope}[shift={(1+\x,-2.28)}, xscale=2, yscale=1.2]
\draw[red] (0,0) arc (0:180:0.5 and 1);  
\draw (0,1.5) arc (360:180:0.5 and 1);
\fill[white] (-0.18,0.73) circle (3pt);  
\fill[white] (-0.82,0.73) circle (3pt);  
\draw[red] (0,0) arc (0:90:0.5 and 1);  
\draw (-0.5,0.5) arc (270:180:0.5 and 1);
\draw(-0.45,0.5)--++(-0.13,0.18);\draw(-0.45,0.5)--++(-0.18,-0.13);
\end{scope}
\begin{scope}[shift={(1+\x,-2.28)}, yscale=-1.2, xscale=2
]
\draw[red] (0,0) arc (0:180:0.5 and 1);  
\draw[red] (0,1.5) arc (360:180:0.5 and 1);
\fill[white] (-0.18,0.73) circle (3pt);  
\fill[white] (-0.82,0.73) circle (3pt);  
\draw[red] (0,0) arc (0:90:0.5 and 1);  
\draw[red] (-0.5,0.5) arc (270:180:0.5 and 1);
\end{scope}}
\foreach\x in {1,2,3}{\fill[black] (6.5+0.2*\x,1) circle (1pt);}
\end{tikzpicture}}}}\\ 
&= \ld' \cdot \theta_{b_{1}} \cdot \hmu'_{g}(\bm{a}, \bm{b}, \bm{j})
= \ld' \cdot \ld \cdot \alpha(\bm{a}) \cdot \theta_{b_{1}}  \cdot \Wbase{g}{a}{b}{j}\,.
\end{split}\]
Let $\eta := \ld' \cdot \ld$, then by \eqref{eq:2-23} we have
\[\RTmat{\cC}{g}(\dt_{\varpi})(\Wbase{g}{a}{b}{j}) = \eta \cdot \theta_{b_{1}} \cdot \Wbase{g}{a}{b}{j}\,.\]
In particular, $\RTmat{\cC}{g}(\dt_{\varpi})$ is a diagonal matrix.
\end{example}

\begin{example}\label{ex:g=1}
Consider the Dehn twist along the curve $m_1$ in Figure \ref{fig:lic-gen}. Similar to the computations in Example \ref{ex:g=2}, there exists a constant $\eta' \in \bC^\times$ such that 
\[\begin{split}
&\RTmat{\cC}{g}(\dt_{m_1})(\Wbase{g}{a}{b}{j})
=\eta' \cdot \theta_{a_{1}}\cdot \Wbase{g}{a}{b}{j}
\end{split}\]
where the second equality follows from \cite[Eq.\ (3.1.6)]{BK01}. 
% In particular, $\RTmat{\cC}{g}(\dt_{m_1})$ is equal up to a scalar to the diagonal matrix whose diagonal entry corresponding to the basis vector $\Wbase{g}{a}{b}{j}$ is the twist value $\theta_{a_{1}}$ of the $a_1$ coordinate of $\Wbase{g}{a}{b}{j}$.

Recall that  the order of the T-matrix $T_{\cC}$ of $\cC$ is finite by Vafa's theorem \cite{Vafa88}, and we set $N:=\ord(T_{\cC})$. Therefore, the above computation implies that
\begin{equation}\label{eq:fin-ord}
\RTrep{\cC}{g}(\dt_{m_1}^N) = [\id]\,.
\end{equation}
Moreover, since the label $a_1$ is arbitrary, we  conclude that the order of $\RTrep{\cC}{g}(\dt_{m_1})$ as an element in $\PGL(\ncg{\cC}{g}, \bC)$ is exactly $N$, i.e., we have
\begin{equation} \label{eq:fin-ord-2}
\ord(\RTrep{\cC}{g}(\dt_{m_1})) = N
\end{equation}
in $\PGL(\ncg{\cC}{g}, \bC)$.    
\end{example}

\begin{example}\label{ex:g=3}
Let $g \ge 3$. Consider the curves $\gamma_{1,g}$ and $\gamma_{2,g}$ bounding a subsurface of genus 1, depicted as follows:
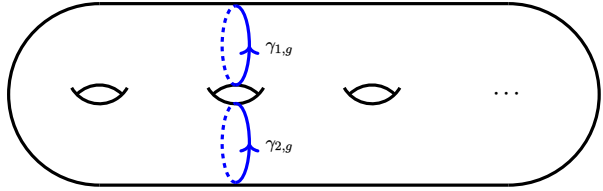
\[\vcenter{\hbox{\scalebox{0.6}{
\begin{tikzpicture}[line width=2pt]
\draw (0,2) arc (90:270:2);
\draw (9, -2) arc (-90:90:2);
\draw (0,2)--(9,2); \draw (0,-2)--(9,-2);
\hole{0,-0.5};\hole{3,-0.5};\hole{6,-0.5};
\node at (9,0) {\Large $\cdots$};
\begin{scope} 
\draw[blue, ->-=0.5] (3, 0.21) arc (-90:90:0.3 and 0.87);
\draw[blue, dashed] (3, 0.21) arc (270:90:0.3 and 0.87);
\node[right] at (3.5,1) {$\gamma_{1, g}$};
\end{scope}
\begin{scope}[yshift=-2.15cm]
\draw[blue, ->-=0.5] (3, 0.21) arc (-90:90:0.3 and 0.87);
\draw[blue, dashed] (3, 0.21) arc (270:90:0.3 and 0.87);
\node[right] at (3.5,1) {$\gamma_{2, g}$};
\end{scope}
\end{tikzpicture}}}}\]
Let $\db_{g} := \dt_{\gamma_{1, g}}\dt_{\gamma_{2, g}}^{-1}$, and consider the following picture:
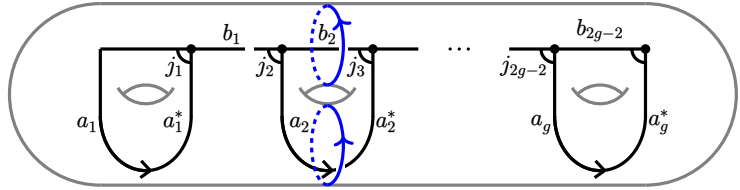
\[\vcenter{\hbox{\scalebox{0.6}{
\begin{tikzpicture}[line width=2pt]
\draw[gray] (-1,2) arc (90:270:2);
\draw[gray] (11, -2) arc (-90:90:2);
\draw[gray] (-1,2)--(11,2); \draw[gray] (-1,-2)--(11,-2);
\begin{scope}[shift={(0,-0.5)}]
\grayhole{0,0}; \grayhole{4,0};\grayhole{10,0};
\end{scope}
% \node at (7,0) {\Large $\cdots$};\node at (7,1) {\Large $\cdots$};
%========================================
% Embedding of basis
%========================================
\begin{scope}[shift={(-1,-1)}]
\draw (0,2)--(7,2); \draw(9,2)--(12,2);
\node[left] at (0.1,0.4) {\Large $a_1^{\phantom{*}}$};
\node[left] at (2,0.4) {\Large $a_1^{*}$};
\node[right] at (4,0.4) {\Large $a_2^{\phantom{*}}$};
\node[right] at (5.9,0.4) {\Large $a_2^{{*}}$};
\node[left] at (10.1,0.4) {\Large $a_g^{\phantom{*}}$};
\node[right] at (11.9,0.4) {\Large $a_g^{{*}}$};
%========================================
\foreach \x in {0,1,2,3,5,6}{\draw (2*\x, 2)--(2*\x,0.5);}
%========================================
\foreach\x in {1,2,3,5,6}
{\begin{scope}[shift={(2*\x, 2)}]
\draw ([shift=(180:0.3cm)]0,0) arc (180:270:0.3cm);
\fill[black] (0,0) circle (3pt);
\end{scope}}
%========================================
\foreach\x in {1,2,3}{\node[left] at (2*\x,1.6) {\Large $j_{\x}$};}
%========================================
\node[left]at(10,1.6){\Large $j_{2g-2}$};
\node[above] at (2.9,1.95) {\Large $b_1$};
\fill[white] (3.2,1.9) rectangle ++(0.185,0.22);
\node[above] at (5,1.95) {\Large $b_2$};
\node[above] at (11,2) {\Large $b_{2g-2}$};	
\end{scope}
%========================================
\foreach\x in {0,4,10}
{\begin{scope}[shift={(1+\x,-2.28)}, xscale=2, yscale=1.2]
\draw (0,1.5) arc (360:180:0.5 and 1);
\draw (-0.5,0.5) arc (270:180:0.5 and 1);
\draw(-0.45,0.5)--++(-0.1,0.18);\draw(-0.45,0.5)--++(-0.11,-0.13);
\end{scope}
\begin{scope}[shift={(1+\x,-2)}, yscale=-1.2, xscale=2
]
\end{scope}}
\foreach\x in {1,2,3}{\fill[black] (6.5+0.2*\x,1) circle (1pt);}
\begin{scope}[shift={(4.25,0.9)}]
\fill[white] (0,0) rectangle ++(0.185,0.22);
\end{scope}
\begin{scope}[shift={(1,0)}] 
\draw[blue, ->-=0.8] (3, 0.21) arc (-90:90:0.35 and 0.87);
\draw[blue, dashed] (3, 0.21) arc (270:90:0.35 and 0.87);
\end{scope}
\begin{scope}[shift={(4.2,-1.75)}]
\fill[white] (0,0) rectangle ++(0.185,0.22);
\end{scope}
\begin{scope}[shift={(1,-2.2)}]
\draw[blue, ->-=0.6] (3, 0.21) arc (-90:90:0.35 and 0.87);
\draw[blue, dashed] (3, 0.21) arc (270:90:0.35 and 0.87);
\end{scope}
\end{tikzpicture}}}}\]
By the above discussions, we can see that there exists a constant $\xi \in \bC^\times$ that is independent of the choice of the basis element $\Wbase{g}{a}{b}{j}$ such that 
\[\RTmat{\cC}{g}(\db_g)(\Wbase{g}{a}{b}{j}) = \xi \cdot \theta^{-1}_{a_{2}}\theta_{b_{2}} \cdot \Wbase{g}{a}{b}{j}\,.\]
In particular, $\RTrep{\cC}{g}(\db_g)$ is diagonal and of finite order.
\end{example}

\subsection{The Torelli group}\label{subsec:torelli}
Let $g \ge 1$ be an integer. Any orientation preserving homeomorphism $f:\Sigma_g \to \Sigma_g$  induces a group isomorphism $f_*: H_1(\Sigma_g, \bZ) \to H_1(\Sigma_g, \bZ)$, and so  $\mcg{g}$ acts  on $H_1(\Sigma_g, \bZ)$ as automorphisms. Since this action preserves the intersection form on $H_1(\Sigma_g, \bZ) \cong \BZ^{2g}$ \cite[Chap.\ 6]{FM12}, the associated representation $\pi_g$ is symplectic, i.e., $\pi_g: \MCG{g} \to \Sp(2g, \bZ)$. By \cite{Bur90}, the image $\pi_g$ is $\Sp(2g, \bZ)$ (see also \cite[Sec.~6.3]{FM12} for modern treatments). The \emph{Torelli group} of $\Sigma_g$, denoted by $\torelli{g}$, is defined as the kernel of  $\pi_g$, and have the  short exact sequence 
\begin{equation}
1 \to \torelli{g} \to \mcg{g} \xrightarrow{\pi_g} \Sp(2g, \bZ) \to 1\,.	
\end{equation}
It is well-known that when $g=1$, $\torelli{1}$ is  trivial, and so $\mcg{1} \cong \SL(2,\bZ)$. However, for $g >1$, $\torelli{g}$ is nontrivial since for any Dehn twist along a separating curve $\gamma$ on $\Sigma_g$, $\dt_\gamma \in \torelli{g}$ by \cite[Sec.~5.6.2]{FM12}. 

The Torelli groups are of particular interest in topology. By \cite{Bir71, Joh79, Pow78}, $\torelli{g}$ was proved to be the normal closure of a single element in $\mcg{g}$ for $g \ge 2$. Such an element is called a \emph{normal generator} of $\torelli{g}$. It is clear that $\torelli{g}$ is generated as a group by the conjugates of a normal generator of $\torelli{g}$ in $\mcg{g}$. In the following, we describe a particular normal generator of $\torelli{g}$, which will be required in next section.

When $g=2$, $\torelli{2}$ is the normal closure of the Dehn twist along any simple closed curve which separates $\Sigma_2$ into two genus 1 sub-surfaces \cite{Pow78}. In this paper, we consider the separating curve $\ITwoGen$  of $\Sigma_2$  depicted in Figure \ref{fig:I2-gen}.
\begin{figure}[ht]
\centering
$\vcenter{\hbox{\scalebox{0.6}{
\begin{tikzpicture}[line width=2pt]
\draw (0,2) arc (90:270:2);
\draw (4, -2) arc (-90:90:2);
\draw (0,2)--(4,2); \draw (0,-2)--(4,-2);
\hole{0,-0.5}; \hole{4,-0.5};
\draw[blue, -<-=0.5] (2,2) arc (90:-90:0.6 and 2);
\draw[blue, dashed] (2,2) arc (90:270:0.6 and 2);
\node at (2.9,-1) {\Large $\ITwoGen$};
\end{tikzpicture}}}}$
\caption{$\torelli{2}$ is normally generated by $\dt_\ITwoGen$.}
\label{fig:I2-gen}
\end{figure}

When $g \ge 3$, $\torelli{g}$ is normally generated by the product $\dt_{\gamma}\dt_{\gamma'}^{-1}$ of Dehn twists, where $\gamma, \gamma'$ is a pair of disjoint, homologous simple closed curves bounding a genus 1 subsurface of $\Sigma_g$ such that neither $\gamma$ nor $\gamma'$ is nullhomologous \cite{Joh79}. In this paper, for any $g \ge 3$, we choose $\gamma_{1, g}$ and $\gamma_{2, g}$ to be the curves bounding the ``left most'' genus 1 subsurface, as is shown in Figure \ref{fig:bp-1} (recall that in this paper surfaces are assumed to be subsets in $\bR^3$ so that their genus are aligned from left to right along the $x$-axis). Then $\db_{g} := \dt_{\gamma_{1, g}}\dt_{\gamma_{2, g}}^{-1}$ normally generates $\torelli{g}$.

\begin{figure}[ht]
\centering
$\vcenter{\hbox{\scalebox{0.6}{
\begin{tikzpicture}[line width=2pt]
\draw (0,2) arc (90:270:2);
\draw (9, -2) arc (-90:90:2);
\draw (0,2)--(9,2); \draw (0,-2)--(9,-2);
\hole{0,-0.5};\hole{3,-0.5};\hole{6,-0.5};
\node at (9,0) {\Large $\cdots$};
\begin{scope} 
\draw[blue, ->-=0.5] (3, 0.21) arc (-90:90:0.3 and 0.87);
\draw[blue, dashed] (3, 0.21) arc (270:90:0.3 and 0.87);
\node[right] at (3.5,1) {$\gamma_{1, g}$};
\end{scope}
\begin{scope}[yshift=-2.15cm]
\draw[blue, ->-=0.5] (3, 0.21) arc (-90:90:0.3 and 0.87);
\draw[blue, dashed] (3, 0.21) arc (270:90:0.3 and 0.87);
\node[right] at (3.5,1) {$\gamma_{2, g}$};
\end{scope}
\end{tikzpicture}}}}$
\caption{A genus  1 bounding pair of curves.}
\label{fig:bp-1}
\end{figure}

\begin{remark}
\label{prop:factor-through-condition}
Let  $\kappa: \MCG{g} \to G$ be any group homomorphism, and $x$ a normal generator of $\torelli{g}$. Then 
$\kappa$ factors through $\Sp(2g, \bZ)$ if and only if $\torelli{g} \subset\ker(\kappa)$, if and only if $x \in \ker(\kappa)$. 
\end{remark}

\begin{definition}
A modular fusion category $\CC$ is called \emph{$g$-Torelli trivial} if $\torelli{g} \subset \ker \RTrep{\CC}{g}$, i.e., $\RTrep{\CC}{g}$ is trivial on $\torelli{g}$.
 \end{definition}

We are interested in the projective representation $\RTrep{\cC}{g} : \MCG{g} \to \PGL(\ncg{\cC}{g})$ with the property $\torelli{g} \subset \ker\RTrep{\cC}{g}$ for a  modular fusion category $\cC$. Let $N := \ord(T_{\cC})$.  An important property of $\RTrep{\cC}{1}$ is that $\ker(\RTrep{\cC}{1})$ is a congruence subgroup of $\SL(2, \bZ)$ of level $N$ \cite[Thm.~6.8]{NS10}. We finish this section by showing that if $\RTrep{\cC}{g}$ factors through $\Sp(2g, \bZ)$ for some $g \ge 2$, then similar congruence property holds.

For any positive integer $m$, let
$$
\Gamma_{m, g} := \ker\left(\Sp(2g, \BZ) \twoheadrightarrow \Sp(2g,\bZ/m\bZ)\right) 
$$
where the underlying map is the natural surjection. The subgroup $\Gamma_{m, g}$  is called a \emph{principal congruence subgroup} of  $\Sp(2g, \BZ)$.  A subgroup of $\Sp(2g, \bZ)$ is called a \emph{congruence subgroup} if it contains $\Gamma_{m, g}$ for some $m \in \bN$, and its \emph{level} is the smallest such $m$.

\begin{thm}
\label{thm:congruence-kernel} Let $\cC$ be a modular fusion category such that $\RTrep{\cC}{g}$ factors through $\s_{\cC,g}:\Sp(2g, \bZ) \to \PGLC$ for some $g \ge 2$. Then $\ker(\s_{\cC,g})$ is a congruence subgroup of $\Sp(2g, \BZ)$ of level $N$, where $N = \ord(T_{\cC})$. Consequently, there exists a group homomorphism $\ol{\rho}_{\cC, g}: \Sp(2g, \bZ/N\bZ) \to \PGL(\ncg{\cC}{g}, \bC)$ such that the following diagram commutes:
\begin{equation}
\xymatrix{
\mcg{g} \ar^-{\RTrep{\cC}{g}}[rr] \ar_-{\pi_g}[d] && \PGL(\ncg{\cC}{g}, \bC) \\
\Sp(2g, \bZ) \ar_-{\s_{\cC, g}}[urr] 
\ar@{->>}[rr]
&&\Sp(2g, \bZ/N\bZ) \ar_-{\ol{\rho}_{\cC, g}}[u]
}    
\end{equation}
In particular, $\RTrep{\cC}{g}$ has finite image. 
\end{thm}

\begin{proof}
By assumption, we have the following commutative diagram 
\begin{equation}
\xymatrix{
\mcg{g} \ar^-{\RTrep{\cC}{g}}[rr] \ar_-{\pi_g}[dr] && \PGL(\ncg{\cC}{g}, \bC) \\
&\Sp(2g, \bZ) \ar_-{\s_{\cC, g}}[ur]&
}    \,.
\end{equation}
According to \cite[Satz 10]{Men65}, $\Gamma_{N, g}$ is the normal closure of the element 
$$(\id_{2g} + N \cdot e_{12}) = (\id_{2g} + e_{12})^N$$ 
in $\Sp(2g, \BZ)$, where $e_{12}$ is the elementary matrix whose (1,2)-entry is 1, and zero elsewhere. It is well-known that the homology classes of the curves $\{m_1, \ell_1, ..., m_g, \ell_g\}$ in Figure \ref{fig:lic-gen} form a symplectic basis of $H_1(\Sigma_g, \bZ)$. Moreover, with respect to this basis, the Dehn twist along $m_1$ is mapped to the matrix $(\id_{2g}+e_{12}) \in \Sp(2g, \bZ)$ via $\pi_g$, i.e., 
\begin{equation}
\label{eq:lifting}
\pi_g(\dt_{m_{1}}) = \id_{2g}+e_{12}\,.
\end{equation}
Therefore, by \eqref{eq:fin-ord}, we have 
\[
\s_{\cC, g}(\id_{2g}+N\cdot e_{12})
=
\s_{\cC, g}(\pi_g(\dt_{m_{1}})^N)
=
\RTrep{\CC}{g}(\dt_{m_{1}}^N)
=
[\id_{\ncg{\cC}{g}}]\,.
\]
As a consequence, $\Gamma_{N, g} \subset \ker(\s_{\cC, g})$. Moreover, by \eqref{eq:fin-ord-2}, $\Gamma_{N, g}$ is the largest congruence subgroup contained in $\ker(\s_{\cC, g})$. Since $\Gamma_{N, g}$ is a finite index subgroup of $\Sp(2g, \BZ)$, the image of $\s_{\cC, g}$ is finite, and so is $\RTrep{\cC}{g}$.
\end{proof}

Comparing to the proof of the congruence property of $\RTrep{\cC}{1}$,  the congruence property in higher genus is almost trivial, since every finite index subgroup of $\Spg$ is congruence for $g > 1$. However, in $\Sp(2,\bZ)= \SL(2,\bZ)$, almost all its finite index subgroup are noncongruence. In view of Lemma \ref{prop:factor-through-condition}, such a congruence property is completely determined by whether the Torelli group $\torelli{g}$ is contained in the kernel of $\RTrep{\cC}{g}$.
We will investigate this question in the next section.

\section{Torelli trivial modular fusion categories} \label{s3}
In this section, we study modular fusion categories $\cC$ which is $g$-Torelli trivial for all $g \ge 1$, which is equivalent to that $\RTrep{\cC}{g}$ factor through $\Sp(2g, \bZ)$ for all $g \ge 1$. We will characterize these modular fusion categories in Theorem \ref{thm:triv-torelli-all-genus}. 

Recall from the previous section,
\begin{itemize}
\item 
$W_{\cC, g}$ is our fixed choice of basis for $\RTspace{\cC}{g}$ (see Lemma \ref{lem:W-basis}) whose graphical presentation is given in \eqref{eq:RTbasis}.

\item 
$\ncg{\cC}{g} = |W_{\cC, g}|$ is the dimension of $\RTspace{\cC}{g}$ (see \eqref{eq:dim-Vg}).

\item The normal generator $\db_g$ of $\torelli{g}$ is given by
$$
\db_g = \left\{\begin{array}{ll}
    \dt_{\ITwoGen} & \text{for $g=2$,  where $\ITwoGen \subset \Sigma_2$ is the curve depicted in Figure \ref{fig:I2-gen};}   \\
    \dt_{\gamma_{1, g}}\dt^{-1}_{\gamma_{2, g}} &\text{for $g \ge 3$, where   $\gamma_{1,g}, \gamma_{2,g} \subset \Sigma_g$ are depicted in Figure \ref{fig:bp-1}}.
\end{array}\right.
$$
\end{itemize}

By Example \ref{ex:g=2}, there exists a constant $\eta \in \bC^\times$ such that 
\begin{equation}\label{eq:g2-action}
\RTmat{\cC}{2}(\dt_{\ITwoGen})(\Wbase{2}{a}{b}{j}) = \eta \cdot \theta_{b_{1}}\cdot \Wbase{2}{a}{b}{j}
\end{equation}
for all basis vector $\Wbase{2}{a}{b}{j} \in W_{\cC, 2}$. Moreover, by Example \ref{ex:g=3}, for any $g \ge 3$, there exits a constant $\xi \in \bC^\times$ such that for any basis vector $\Wbase{g}{a}{b}{j} \in W_{\cC, g}$
\begin{equation}\label{eq:g3-action}
\RTmat{\cC}{3}(\db_g)(\Wbase{g}{a}{b}{j}) = \xi \cdot \theta_{a_2}^{-1}\theta_{b_2} \cdot \Wbase{g}{a}{b}{j}\,.    
\end{equation}

Recall that $\RTspace{\cC}{g} = \cC(\1, I^{\ot g})$ where $I = \bigoplus_{a \in \irr(\cC)} a \ot a^*$. Define 
\begin{equation}
\label{eq:I1}
\AdSet{1}_{\cC} := \bigcup_{b \in \Irr(\CC)}\{a \in \Irr(\CC) \mid N_{b, b^*}^{a} > 0\}
\end{equation}
and 
\[\AdSet{k+1}_{\cC} := \bigcup_{b, c \in \AdSet{k}} \{a \in \Irr(\CC) \mid N_{b, c}^{a} > 0\} \]
for $k \ge 1$. We will omit the subscript and simply write $\AdSet{k}$ when  the context is clear. It is easy to see that $\AdSet{k} \subset \AdSet{k+1}$, and $\AdSet{k}$ is nothing but the set of simple summands of $I^{\ot k}$ for all $k \ge 1$.

\subsection{Torelli trivial modular fusion category for all genus}\label{subsec:g>=3} 

In this section, we classify modular fusion categories whose genus $g$ mapping class group representation has $\torelli{g}$ in its kernel for all $g \ge 1$.

\begin{lemma}\label{lem:g3-imply-g2}
Let $\cC$ be a modular fusion category. If $\torelli{3} \subset \ker (\RTrep{\cC}{3})$, then $\torelli{2} \subset \ker (\RTrep{\cC}{2})$.
\end{lemma}
\begin{proof}
First, consider the separating curve $\ITwoGen' \subset \Sigma_3$ in Figure \ref{fig:I3-to-I2}. By Example \ref{ex:g=2}, there exists a constant $\eta' \in \bC^\times$ such that
\begin{equation}\label{eq:g3-separating}
\RTmat{\cC}{3}(\dt_{\ITwoGen'})(\Wbase{3}{a}{b}{j}) = \eta' \cdot \theta_{b_{1}} \cdot \Wbase{3}{a}{b}{j}
\end{equation}
for all basis vector $\Wbase{3}{a}{b}{j} \in W_{\cC, 3}$. 
\begin{figure}[ht]
\centering
$\vcenter{\hbox{\scalebox{0.6}{
\begin{tikzpicture}[line width=2pt]
\draw (0,2) arc (90:270:2);
\draw (7, -2) arc (-90:90:2);
\draw (0,2)--(7,2); \draw (0,-2)--(7,-2);
\hole{0,-0.5}; \hole{4,-0.5}; \hole{7,-0.5};
\draw[blue, -<-=0.5] (2,2) arc (90:-90:0.6 and 2);
\draw[blue, dashed] (2,2) arc (90:270:0.6 and 2);
\node at (2.9,-1) {\Large $\ITwoGen'$};
\end{tikzpicture}}}}$
\caption{$\dt_{\ITwoGen'} \in \torelli{3}$.}
\label{fig:I3-to-I2}
\end{figure}

Then note that we have an injection of vector spaces $U: \RTspace{\cC}{2} \hookrightarrow \RTspace{\cC}{3}$ extending the following injective map on the basis vectors:
\[\begin{split}
W_{\cC, 2} &\to W_{\cC, 3}\,, \\
\hmu_{2}((a_1, a_2), (b_1, b_2), (j_1, j_2)) &\mapsto 
\hmu_{3}((a_1, a_2, \1), (b_1, b_2, \1, \1), (j_1, j_2, 1, 1))\,.
\end{split}\]
(The admissibility condition implies that $b_2 = a_2$, but this does not have any impact on our argument below.) Pictorially, the above map is given by 

\begin{figure}[ht]
\centering
$\begin{aligned}
\vcenter{\hbox{\scalebox{0.6}{
\begin{tikzpicture}[line width=2pt]
\begin{scope}[shift={(-1,-1)}]
\draw (0,2)--(6,2);
\node[below] at (0,0) {\Large $a_1^{\phantom{*}}$};
\node[below] at (2,0) {\Large $a_1^{*}$};
\node[below] at (4,0) {\Large $a_2^{\phantom{*}}$};
\node[below] at (6,0) {\Large $a_2^{{*}}$};
%========================================
\foreach \x in {0,1,2,3}{\draw (2*\x, 2)--(2*\x,0);}
%========================================
\foreach\x in {1,2,3}
{\begin{scope}[shift={(2*\x, 2)}]
\draw ([shift=(180:0.3cm)]0,0) arc (180:270:0.3cm);
\fill[black] (0,0) circle (3pt);
\end{scope}}
%========================================
\foreach\x in {1,2}{\node[left] at (2*\x-0.05,1.6) {\Large $j_{\x}$};}
%========================================
\node[above] at (3,2) {\Large $b_1$};
\node[above] at (5,2) {\Large $b_2$};
\end{scope}
\end{tikzpicture}}}}
&\mapsto 
\vcenter{\hbox{\scalebox{0.6}{
\begin{tikzpicture}[line width=2pt]
\begin{scope}[shift={(-1,-1)}]
\draw (0,2)--(6,2);
\draw[dashed] (6,2)--(10,2); 
\node[below] at (0,0) {\Large $a_1^{\phantom{*}}$};
\node[below] at (2,0) {\Large $a_1^{*}$};
\node[below] at (4,0) {\Large $a_2^{\phantom{*}}$};
\node[below] at (6,0) {\Large $a_2^{{*}}$};
\node[below] at (8,0) {\Large $\1$};
\node[below] at (10,0) {\Large $\1$};
%========================================
\foreach \x in {0,1,2,3}{\draw (2*\x, 2)--(2*\x,0);}
\foreach \x in {4,5}{\draw[dashed] (2*\x, 2)--(2*\x,0);}
%========================================
\foreach\x in {1,2,3}
{\begin{scope}[shift={(2*\x, 2)}]
\draw ([shift=(180:0.3cm)]0,0) arc (180:270:0.3cm);
\fill[black] (0,0) circle (3pt);
\end{scope}}
%========================================
\foreach\x in {1,2}{\node[left] at (2*\x-0.05,1.6) {\Large $j_{\x}$};}
%========================================
\node[above] at (3,2) {\Large $b_1$};
\node[above] at (5,2) {\Large $b_2$};
\node[above] at (7,2) {\Large $\1$};
\node[above] at (9,2) {\Large $\1$};	
\end{scope}
\end{tikzpicture}}}}
%\\
%\hmu_{2}((a_1, a_2), (b_1, b_2), (j_1, j_2)) & \mapsto 
%\hmu_{3}((a_1, a_2, \1), (b_1, b_2, \1, \1), (j_1, j_2, 1, 1))
\end{aligned}
$
\caption{The map $U: \RTspace{\cC}{2} \to \RTspace{\cC}{3}$ on basis vectors.}
\label{fig:two_to_three}
\end{figure}
By comparing \eqref{eq:g2-action} and \eqref{eq:g3-separating}, we immediately have that the following diagram is commutative up to a nonzero scalar:
\[
\begin{tikzcd}[row sep=large] 
\RTspace{\cC}{2} \ar[r, hook, "U"] \ar[d, "{\RTmat{\cC}{2}(\dt_{\ITwoGen})}"'] & \RTspace{\cC}{3} \ar[d, "\RTmat{\cC}{3}(\dt_{\ITwoGen'})"]\\
\RTspace{\cC}{2} \ar[r, hook, "U"] & \RTspace{\cC}{3}
\end{tikzcd} \,.
\]
Since $\ITwoGen'$ is a separating curve, the Dehn twist $\dt_{\ITwoGen'}$ is contained in $\torelli{3}$ \cite[Sec.~5.6.2]{FM12}. By assumption, $\RTrep{\cC}{3}(\dt_{\ITwoGen'}) = [\RTmat{\cC}{3}(\dt_{\ITwoGen'})] = [\id]$. It follows from the commutative diagram that $\RTrep{\cC}{2}(\dt_{\ITwoGen}) = [\RTmat{\cC}{2}(\dt_{\ITwoGen})] = [\id]$.
\end{proof}

\begin{lemma}\label{lem:g2}
Let $\cC$ be a modular fusion category. Then $\torelli{2} \subset \ker(\RTrep{\cC}{2})$  if and only if $\theta_{x} = 1$ for all $x \in \AdSet{1}$.
\end{lemma}
\begin{proof}
Note that every $x \in \AdSet{1}$ can be the $b_1$-component of some of our chosen basis vectors in $W_{\cC, 2}$. Indeed, by definition, $N_{a, a^*}^{x} \ge 1$ for some $a \in \irr(\cC)$, and so $((a, a), (x, a), (1, 1))$ is admissible. If $\torelli{2} \subset \ker(\RTrep{\cC}{2})$, then  \eqref{eq:g2-action} imply that $\theta_{x} = \theta_{y}$ for all $x, y \in \AdSet{1}$. In particular, since $\1 \in \AdSet{1}$, we have $\theta_{x} =\theta_{\1} = 1$ for all $x \in \AdSet{1}$. The converse follows immediately from \eqref{eq:g2-action}.
\end{proof}

Recall that a simple object $a$ in a fusion category $\cF$ is called \emph{invertible} if $a \ot a^* \cong \1$. A fusion category is called \emph{pointed} if all of its simple objects are invertible. It is easy to see that the set of all invertible objects in an arbitrary fusion category $\cF$  naturally forms a group under the tensor. The maximal pointed fusion subcategory of $\cF$ is denoted by $\cF_{\pt}$. The following result is a refinement of \cite[Prop.\ 6.7]{MW25}.

\begin{thm}\label{thm:triv-torelli-all-genus}
Let $\CC$ be a modular fusion category.  The following statements are equivalent.
\begin{enumerate}
\item
    $\RTrep{\cC}{g}$ factors through $\Sp(2g, \bZ)$ for all $g \ge 1$.
\item 
    $\RTrep{\cC}{g}$ factors through $\Sp(2g, \bZ)$ for all $g \ge 3$.
\item 
    $\RTrep{\cC}{g}$ factors through $\Sp(2g, \bZ)$ for some $g \ge 3$.
\item
    $\RTrep{\cC}{3}$ factors through $\Sp(6, \bZ)$. 
\item
    $\cC$ is pointed.
\end{enumerate} 
\end{thm}

\begin{proof}
Let $g \ge 3$. A pair of simple objects $(x, y)$ appears as $(a_2,b_2)$ in any basis vector $\Wbase{g}{a}{b}{j} \in W_{\cC,g}$ presented in \eqref{eq:W-basis}), called $(2,2)$-realizable in $W_{\cC,g}$,  if and only if the following conditions are satisfied:
$$
\cC(y, I \ot x) \ne 0 \quad \text{ and } \quad \cC(\1, y\ot x^* \ot I^{\ot (g-2)}) \ne 0.
$$
Since $I$ is self-dual and is a summand of $I^{\o (g-2)}$ as $g \ge 3$, we find 
\[\cC(\1, y\ot x^* \ot I^{\ot (g-2)}) \cong  \cC(I^{\ot (g-2)}, y\ot x^*) \cong \cC(y, I^{\ot (g-2)} \ot x)\]
contains a subspace isomorphic to $\cC(y, I \ot x)$. Thus, $(x,y)$ is (2,2)-realizable in $W_{\cC,g}$ if and only if
$$
0 \ne \cC(y, I\ot x) \cong \cC(y\ot x^*, I)
$$
which is independent of $g$. Therefore, $(x,y)$ is $(2,2)$-realizable in $W_{\cC,g}$ for some $g \ge 3$ if and only if it is $(2,2)$-realizable in $W_{\cC,g}$ for all $g \ge 3$.  In particular, $(\1,\1)$ is $(2,2)$-realizable in $W_{\cC,g}$. 
%for all $g \ge 3$.
% For any $x \in \irr(\cC)$, we define
% \[\Ld_{x} :=  \{y \in \Irr(\cC) \mid \cC(y, I \ot x) \ne 0\}\,.\] 
% Then the set of values of the $b_2$-coordinate of an admissible element, denoted by $\Ld$, satisfies $\Ld = \bigcup_{x \in \irr(\cC)} \Ld_{x}$. 

Recall from \eqref{eq:g3-action} that there exists a constant $\xi \in \bC^\times$ such that for any $\Wbase{g}{a}{b}{j} \in W_{\cC,g}$, we have 
\[\RTmat{\cC}{g}(\db_g)(\Wbase{g}{a}{b}{j}) = \xi \cdot \theta_{a_{2}}^{-1}\theta_{b_{2}} \cdot \Wbase{g}{a}{b}{j}\,.\]
Therefore, for any $g \ge 3$, $\RTrep{\cC}{g}(\db_g) = [\RTmat{\cC}{g}(\db_g)] = [\id]$  if and only if 
\begin{equation} \label{eq:trivial_condition}
\theta_x^{-1} \theta_y = \theta_\1^{-1} \theta_\1 = 1, \quad \text{or equivalently,} \quad \theta_x = \theta_y
\end{equation}
for all $(x,y)$ which is (2,2)-realizable in $W_{\cC,g}$. It follows from the last paragraph and Lemma \ref{prop:factor-through-condition} that (b), (c) and (d) are equivalent.

 Clearly, (a) implies (b)-(d). By Lemma \ref{lem:g3-imply-g2} and the fact that $\torelli{1}$ is trivial, we can conclude that (a)-(d) are equivalent.

(e) $\Rightarrow$ (d) If $\cC$ is pointed, then $I \cong \ncg{\CC}{1} \1$. Thus, $(x,y)$ is (2,2)-realizable in $W_{\cC,3}$ if and only if $x=y$. Consequently, $\theta_x = \theta_y$ for any $(x,y)$ which is (2,2)-realizable in $W_{\cC,3}$, and so (d) follows.

(d) $\Rightarrow$ (e) Assume  $\RTrep{\cC}{3}$ factors through $\Sp(6, \bZ)$, and so $\RTrep{\cC}{3}(\db_3) = [\id]$. By Lemmas \ref{lem:g3-imply-g2} and \ref{lem:g2}, for any $z \in I^{(1)}$, $\theta_{z}=1$. For any $x \in \irr(\CC)$,  $z\ot x = \sum_{y \in \irr(\cC)} N_{z,x}^y y$  implies $(x,y)$ is (2,2)-realizable whenever $N_{z,x}^y \ne 0$. Thus, \eqref{eq:trivial_condition} implies that $\theta_y = \theta_x$ for any simple subobject $y$ of $z\ot x$. By the twist equation, we have 
\begin{align*}
\theta_{x}(\be_{x, z}\circ \be_{z,x}) & = \theta_{z}\theta_{x}(\be_{x,z}\circ\be_{z,x}) = \theta_{z \ot x} = 
\sum_{y \in \irr(\cC)} \sum_{j =1}^{N_{z, x}^{y}} \theta_{y} \cdot \hmu_{j}(y |z, x) \circ \cmu_{j}(z, x|y)\\
&= \theta_{x} \sum_{y \in \irr(\cC)} \sum_{j =1}^{N_{z, x}^{y}}  \hmu_{j}(y |z, x) \circ \cmu_{j}(z, x|y) = \theta_x \,\id_{z \o x}
\end{align*}
Since $\theta_x \ne 0$, we have $\be_{x, z}\circ\be_{z,x} = \id_{z \ot x}$ for all $z \in \AdSet{1}$ and $x \in \irr(\cC)$. Consequently, $\AdSet{1}$ is in the M\"uger center of $\cC$ (see Section \ref{subsec:mtc}). The modularity $\cC$ implies $\AdSet{1} = \{\1\}$,  which means  $a \ot a^* \cong \1$ for any $a \in \irr(\cC)$. Therefore, $\cC$ is pointed.
\end{proof}

By Theorem \ref{thm:triv-torelli-all-genus} above, for any modular fusion category $\cC$,  exactly one of the following three conditions holds:

\begin{enumerate}[label=(\Roman*)]
\item $\RTrep{\cC}{g}$ factors through $\Sp(2g, \bZ)$ for all $g \ge 1$, or equivalently, $\cC$ is pointed;
\item $\RTrep{\cC}{2}$ factors through $\Sp(4, \bZ)$, but $\RTrep{\cC}{g}$ does not factor through $\Sp(2g, \bZ)$ for any $g \ge 3$;
\item $\RTrep{\cC}{g}$ does not factor through $\Sp(2g, \bZ)$ for any $g \ge 2$.
\end{enumerate}

There are plenty of modular fusion categories satisfying (III) above. Indeed, by Lemma \ref{lem:g2} and Theorem \ref{thm:triv-torelli-all-genus}, as long as a modular fusion category $\cC$ satisfies $\theta_{x} \ne 1$ for some $x \in \AdSet{1}$, then it satisfies (III). The simplest example  is  the quantum group modular fusion category $\cC = \PSU(2)_{3}$ (we adopt the conventions in \cite{NWZ-trans}). In fact, $\cC$ is a Fibonacci modular fusion category with $\irr(\cC)=\{\1, x\}$ and the fusion rule $x \ot x = \1 \oplus x$, $d_x = \frac{1+\sqrt{5}}{2}$, $\theta_{x} = \zeta_5^2$ and $I^{(1)} = \irr(\cC)$. Therefore, the mapping class group representations of $\cC = \PSU(2)_{3}$ cannot factor through $\Sp(2g, \bZ)$ for any $g \ge 2$.

Now we give an example of a modular fusion category that satisfies (II) above. Recall that the \emph{adjoint subcategory} $\cF_{\ad}$ of a fusion category $\cF$ is the full fusion subcategory of $\cF$ generated by objects in $\AdSet{1}_{\cF}$ \cite[Sec.~3]{GN08}. It is easy to see that $\AdSet{1}_{\cF} \subset \Irr(\cF_{\ad}) = \AdSet{n}_{\cF}$ for some integer $n$.

\begin{example}\label{ex:DwD8}
Let $\cC = \zdr(\Vec_{E_{8}}^\w)$ be the Drinfeld center of the pointed fusion category $\Vec_{E_{8}}^\w$, where $E_8$ is the elementary 2-group of order 8, and $\omega$ is a normalized 3-cocycle of weight 1 of $E_8$ (e.g., the alternating tricharacter of $E_8$, see \cite[Sec.\ 5]{GMN07} for details). The modular data of $\cC$ is given as follows:
\[
S = \left(
\begin{smallmatrix}
1 & 1 & 1 & 1 & 1 & 1 & 1 & 1 & 2 & 2 & 2 & 2 & 2 & 2 & 2 & 2 & 2 & 2 & 2 & 2 & 2 & 2 \\
1 & 1 & 1 & 1 & 1 & 1 & 1 & 1 & 2 & 2 & 2 & 2 & -2 & -2 & 2 & 2 & -2 & -2 & -2 & -2 & -2 & -2 \\
1 & 1 & 1 & 1 & 1 & 1 & 1 & 1 & 2 & 2 & -2 & -2 & 2 & 2 & -2 & -2 & 2 & 2 & -2 & -2 & -2 & -2 \\
1 & 1 & 1 & 1 & 1 & 1 & 1 & 1 & 2 & 2 & -2 & -2 & -2 & -2 & -2 & -2 & -2 & -2 & 2 & 2 & 2 & 2 \\
1 & 1 & 1 & 1 & 1 & 1 & 1 & 1 & -2 & -2 & 2 & 2 & 2 & 2 & -2 & -2 & -2 & -2 & 2 & 2 & -2 & -2 \\
1 & 1 & 1 & 1 & 1 & 1 & 1 & 1 & -2 & -2 & 2 & 2 & -2 & -2 & -2 & -2 & 2 & 2 & -2 & -2 & 2 & 2 \\
1 & 1 & 1 & 1 & 1 & 1 & 1 & 1 & -2 & -2 & -2 & -2 & 2 & 2 & 2 & 2 & -2 & -2 & -2 & -2 & 2 & 2 \\
1 & 1 & 1 & 1 & 1 & 1 & 1 & 1 & -2 & -2 & -2 & -2 & -2 & -2 & 2 & 2 & 2 & 2 & 2 & 2 & -2 & -2 \\
2 & 2 & 2 & 2 & -2 & -2 & -2 & -2 & 4 & -4 & 0 & 0 & 0 & 0 & 0 & 0 & 0 & 0 & 0 & 0 & 0 & 0 \\
2 & 2 & 2 & 2 & -2 & -2 & -2 & -2 & -4 & 4 & 0 & 0 & 0 & 0 & 0 & 0 & 0 & 0 & 0 & 0 & 0 & 0 \\
2 & 2 & -2 & -2 & 2 & 2 & -2 & -2 & 0 & 0 & 4 & -4 & 0 & 0 & 0 & 0 & 0 & 0 & 0 & 0 & 0 & 0 \\
2 & 2 & -2 & -2 & 2 & 2 & -2 & -2 & 0 & 0 & -4 & 4 & 0 & 0 & 0 & 0 & 0 & 0 & 0 & 0 & 0 & 0 \\
2 & -2 & 2 & -2 & 2 & -2 & 2 & -2 & 0 & 0 & 0 & 0 & 4 & -4 & 0 & 0 & 0 & 0 & 0 & 0 & 0 & 0 \\
2 & -2 & 2 & -2 & 2 & -2 & 2 & -2 & 0 & 0 & 0 & 0 & -4 & 4 & 0 & 0 & 0 & 0 & 0 & 0 & 0 & 0 \\
2 & 2 & -2 & -2 & -2 & -2 & 2 & 2 & 0 & 0 & 0 & 0 & 0 & 0 & 4 & -4 & 0 & 0 & 0 & 0 & 0 & 0 \\
2 & 2 & -2 & -2 & -2 & -2 & 2 & 2 & 0 & 0 & 0 & 0 & 0 & 0 & -4 & 4 & 0 & 0 & 0 & 0 & 0 & 0 \\
2 & -2 & 2 & -2 & -2 & 2 & -2 & 2 & 0 & 0 & 0 & 0 & 0 & 0 & 0 & 0 & 4 & -4 & 0 & 0 & 0 & 0 \\
2 & -2 & 2 & -2 & -2 & 2 & -2 & 2 & 0 & 0 & 0 & 0 & 0 & 0 & 0 & 0 & -4 & 4 & 0 & 0 & 0 & 0 \\
2 & -2 & -2 & 2 & 2 & -2 & -2 & 2 & 0 & 0 & 0 & 0 & 0 & 0 & 0 & 0 & 0 & 0 & 4 & -4 & 0 & 0 \\
2 & -2 & -2 & 2 & 2 & -2 & -2 & 2 & 0 & 0 & 0 & 0 & 0 & 0 & 0 & 0 & 0 & 0 & -4 & 4 & 0 & 0 \\
2 & -2 & -2 & 2 & -2 & 2 & 2 & -2 & 0 & 0 & 0 & 0 & 0 & 0 & 0 & 0 & 0 & 0 & 0 & 0 & -4 & 4 \\
2 & -2 & -2 & 2 & -2 & 2 & 2 & -2 & 0 & 0 & 0 & 0 & 0 & 0 & 0 & 0 & 0 & 0 & 0 & 0 & 4 & -4
\end{smallmatrix}
\right)
\]
and $T = \diag(1, 1, 1, 1, 1, 1, 1, 1, 1, -1, 1, -1, 1, -1, 1, -1, 1, -1, 1, -1, -i, i)$. Clearly, $\cC$ is not pointed, and so $\RTrep{\cC}{g}$ does not factor through $\Sp(2g, \bZ)$ for any $g \ge 3$. Moreover, 
$$
\cC_{\ad} = \cZ_2(\cC_{\pt},\cC) = \cC_{\pt}
$$
and so $\theta_x =1$ for $x \in \AdSet{1}$. By Lemma \ref{lem:g2}, $\RTrep{\cC}{2}$ factors through $\Sp(4, \bZ)$. Therefore, $\zdr(\Vec_{E_{8}}^\w)$ is an example satisfying (II) above.
\end{example}
 
If $\cF$ is a ribbon fusion category with the ribbon structure $\theta$, a fusion subcategory  $\cE$  of $\cF$, then $\cE$ is called \emph{isotropic} if $\theta_x =1$ for $x \in \irr(\cE)$.  

The modular fusion category $\cC=\zdr(\Vec_{E_{8}}^\w)$ in Example \ref{ex:DwD8} satisfies $\AdSet{1} = \irr(\cC_{\ad})=\irr(\cC_{\pt})$.  By reading the T-matrix, $\cC_{\ad}$ is isotropic. Note also that all pointed modular fusion categories have isotropic adjoint subcategories, as they are equivalent to $\Vec$.
\begin{cor} \label{c:iso_adj}
Let $\cC$ be a non-pointed modular fusion category with isotropic adjoint subcategory. Then $\RTrep{\cC}{2}$ factors through $\Sp(4, \bZ)$ but 
$\RTrep{\cC}{g}$ does not factor through $\Sp(2g, \bZ)$ for all $g >2$.
\end{cor}
\begin{proof}
It follows immediately from the definition, Lemma \ref{lem:g2} and Theorem \ref{thm:triv-torelli-all-genus}.
\end{proof}

\subsection{Modular fusion categories with isotropic adjoint}\label{subsec:g=2}
% In Section \ref{subsec:g>=3}, we classified modular fusion categories such that all of its mapping class group representations factor through the corresponding symplectic groups. A related question is: Is there a modular fusion category $\cC$ such that $\RTrep{\cC}{g}$ factors through $\Sp(2g, \bZ)$ for some values of $g$, and not for others? 

Corollary \ref{c:iso_adj} provides a motivation for studying modular fusion categories with isotropic adjoint.
 In this section, we establish a characterization of these modular fusion categories.

We start by showing that the pivotal dimensions of the objects in the adjoint subcategory are independent of the pivotal structures of the underlying fusion category, which could be known to some experts. We include a proof here for completeness. 

\begin{lemma}\label{lem:j-ad}
Let $\FF$ be fusion category with pivotal structures $\piv, \piv'$. Then for any $v \in \FF_{\ad}$, we have $(\piv^{-1}\piv')_v = \id_v$ and $d^\piv_v = d^{\piv'}_v$.
\end{lemma}
\begin{proof}
Let $\fA := \Aut_{\ot}(\id_{\FF})$ be the group of automorphisms of the monoidal functor $\id_{\FF}$. Then $\fA$ is an abelian group by \cite{GN08}, and we denote its dual group by $\hat{\fA}$. According to the proof of \cite[Prop.~3.9]{GN08}, $\hat{\fA}$ gives rise to a faithful grading $\FF = \bigoplus_{\chi\in \hat{\fA}} \FF_\chi$ by setting
\[
\irr(\FF_\chi) := \{v \in \irr(\FF) \mid \forall \Phi \in \fA,\ \Phi_v = \chi(\Phi) \cdot\id_v\}
\]
for any $\chi \in \hat{\fA}$. By definition, $\FF_\ad = \FF_\e$, where $\e$ is the trivial character on $\fA$. In other words, for any $v \in \FF_\ad$ and any $\Phi\in \fA$, $\Phi_v = \id_v$.

Since $\piv^{-1}\piv' \in \fA$, for any $v \in \FF_\ad$, we have $(\piv^{-1}\piv')_v = \id_v$. Therefore, we have
\[
d^\piv_v =  \ev_{v^*}(\piv_v \ot \id_{v*})\coev_v
= \ev_{v^*}(\piv'_v \ot \id_{v*})\coev_v = d^{\piv'}_v
\]
for any $v \in \FF_\ad$.
\end{proof}

 Let $\cF$ be a braided fusion category with braiding $\be$, and let $u: \id_{\cF} \to (-)^{**}$ be the \emph{Drinfeld isomorphism} defined by
\[u_a := (\ev_{a} \ot \id_{a^{**}})\circ (\id_{a^*} \ot \be_{a^{**}, a})\circ (\coev_{a^*} \ot \id_a)\]
for $a \in \cF$. Then for any pivotal structure $\piv$ on $\cF$, the natural isomorphism $\theta^{\piv}: \id_{\cC} \to \id_{\cC}$ defined by 
\begin{equation}\label{eq:theta-piv}
\theta^{\piv}_a := \piv_{a}^{-1}\circ u_a \quad\text{for all $a \in \cF$}
\end{equation}
is a twist on $\cF$ (see \cite[Def.~8.10.1]{EGNO} or \cite[p.38]{NS07}), i.e., $\theta^\piv$ satisfies the twist equation \eqref{eq:tw-def}. We call $\theta^{\piv}$ the twist associated to $\piv$. The equation \eqref{eq:theta-piv} defines a one-to-one correspondence between twists and pivotal structures of $\cF$. It is well-known that $\piv$ is spherical if and only if  $\theta^{\piv}$ is a ribbon structure. Thus, if $\cF$ is a ribbon fusion category with ribbon structure $\theta$, then $\piv^{\theta} := u \circ \theta^{-1}$ is a spherical structure on $\cF$, and we call $\piv^{\theta}$ the spherical structure associated to $\theta$. In the following, when not specified, we will always assume that a ribbon fusion category is equipped with the associated spherical pivotal structure, and we will write the corresponding pivotal dimension of an object $x$  as $d_x$. 

One can define the categorical trace of the Drinfeld isomorphism $u_a : a \to a^{**}$ as $\tr(u_a) := \ev_{a^*}(u_a \o \id)\coev_a$. For any simple object $a$, we find
\begin{equation}\label{eq:tr_drin}
d^{\piv}_a \theta^{\piv}_{a} = \tr^{\piv}(\theta^{\piv}_{a}) 
= \ev_{a^*}  ((\piv_{a} \circ \theta^{\piv}_a) \ot\id_{a^*}) \coev_{a}
= \ev_{a^*} (u_a \ot\id_{a^*})  \coev_{a} =\tr(u_a)
\end{equation}
which is independent of pivotal structure $\piv$ of $\cF$. 

A fusion category $\cF$ is called \emph{pseudounitary} if $\FPdim(\cF) = \dim(\cF)$. By \cite[Prop.\ 8.23]{ENO05}, a pseudounitary fusion category admits a canonical spherical structure whose corresponding pivotal dimension of any object is its Frobenius-Perron dimension. We call such spherical structure the \emph{canonical spherical structure} of a pseudounitary fusion category.  In particular, a pseudounitary braided fusion category has a canonical ribbon structure associated to its canonical spherical structure.

\begin{cor}\label{cor:pu-ad}
Let $\cF$ be a pseudounitary fusion category. For any pivotal structure $\piv$ on $\FF$, we have $d^\piv_v = \FPdim(v)>0$ for any $v \in \FF_\ad$. If, in addition,  $\FF$ is braided, then for any other pivotal structure $\piv'$ on $\cF$, we have $\theta^{\piv}_v = \theta^{\piv'}_v$ for all $v \in \cF_{\ad}$.
\end{cor}
\begin{proof}
Let $\piv_0$ be the canonical spherical structure of $\cF$. By Lemma \ref{lem:j-ad}, for any $v \in \FF_\ad$, we have $d^\piv_v = d^{\piv_{0}}_v = \FPdim(v)$. If, in addition, $\FF$ is braided, then  for any $a \in \irr(\cF)$, we have
\[
d^{\piv}_a \cdot \theta^{\piv}_{a}  =\tr(u_a) = d^{\piv_0}_a \cdot \theta^{\piv_0}_{a}
\]
by \eqref{eq:tr_drin}, which implies $\theta^{\piv}_{v} = \theta^{\piv_0}_v$ for $v \in \irr(\cF_\ad)$. Now, our last assertion follows.
\end{proof}

In general,  a braided fusion category is isotropic for a ribbon structure, and is not isotropic for another ribbon structure. For example, the braided fusion category $\Rep(\bZ/2\bZ)$ has 2 ribbon structures, and it is isotropic when equipped with the canonical ribbon structure but it is not isotropic when equipped with the other.

However, Corollary \ref{cor:pu-ad} shows that, if the adjoint subcategory $\cC_{\ad}$ of a modular fusion category $\cC$ is isotropic, then every ribbon structure on $\cC$ restricts to the identity on $\cC_{\ad}$, i.e., $\cC_{\ad}$ is isotropic with respect to any ribbon structure on $\cC$.

Now we briefly review some properties of isotropic fusion subcategories of ribbon fusion categories. Let $\cC$ be a ribbon fusion category with ribbon structure $\theta$, and let $\cE$ be an isotropic fusion subcategory of $\cC$. By the twist equation \eqref{eq:tw-def}, $\cE$ is symmetric. Hence, by Deligne's theorems \cite{Del02}, there exists a finite group $G$ such that $\cE$ is braided equivalent to either $\Rep(G)$ or $\Rep(G, z)$, where $z \in G$ is a central element of order 2 which endows  $\Rep(G)$ with a braiding, see op.~cit. for more details. 
We call $\cE$ Tannakian (resp.~super Tannakian) if $\cE$ is braided equivalent to $\Rep(G)$ (resp.~$\Rep(G, z)$). 
 
\begin{lemma} \label{cor:Tan_cri}
Let $\cE$ be a symmetric fusion category equipped with a spherical structure $\piv$. Then $\cE$ is Tannakian if and only if $d_a^\piv \theta_a^\piv > 0$ for all $a \in \irr(\cE)$.
\end{lemma}
\begin{proof}
Since $\FPdim(\cE)\in \bZ$, $\cE$ is pseudounitary. Denote the canonical spherical structure on $\cE$ by $\piv_0$, and the corresponding ribbon structure by $\theta^{\piv_0}$. Then $\cE$ is Tannakian if and only if $\theta_a^{\piv_0} = 1$ for all $a \in \irr(\cE)$, which is equivalent to
$$
d_a^\piv \theta_a^\piv = d_a^{\piv_0} \theta_a^{\piv_0} = d_a^{\piv_0} > 0
$$
by Corollary \ref{cor:pu-ad}.
\end{proof}

Recall the notion of minimal modular extension of a ribbon fusion category from \cite{Mug03S}.

\begin{lemma} \label{l:integral}
Let $\CC$ be a modular fusion category and let $G$ be its group of invertible objects. If $\CC_\ad$ is isotropic, then $\cC$ is pseudounitary, and $\cC_{\ad}$ is braided equivalent to $\Rep(\hat{H})$, where $H$ is the M\"uger centralizing subgroup of $G$, i.e., $H = \Irr(\cZ_2(G, \cC_\pt))$. In particular, $\CC$ is a minimal modular extension of the pointed ribbon category $\CC(G, q)$ where the quadratic form $q: G \to \BC^\times$ is given by $q(g) = \theta_g$. 
\end{lemma}
\begin{proof}
Since $\CC_\ad$ is isotropic, the twist equation \eqref{eq:tw-def} implies $\CC_\ad$ is symmetric. 
Hence, by Deligne's theorems \cite{Del02}, $\CC_\ad$ is either Tannakian or super-Tannakian. In particular, $\FPdim(\CC_{ad}) \in \bZ$.  Since $\cC$ is modular,
$\cZ_2(\CC_\ad, \cC)=\CC_\pt$ by \cite[Cor.\ 3.27]{DGNO-BFC}, and so $\CC_\ad \subset \CC_\pt$. Thus, $\CC_\ad$ is also pointed, and $H:= \irr(\CC_{\ad})$ is a subgroup of $G$. By \cite[Cor.~3.14]{DGNO-BFC}, we have 
\[\FPdim(\CC) = \FPdim(\CC_\ad)\FPdim(\CC_\pt) = |H| \cdot |G| \in \BZ\,.\] 
Thus, $\cC$ is pseudounitary by \cite[Prop.~8.24]{ENO05}, and a minimal modular extension of $\CC(G, q)$. It follows from  Corollary \ref{cor:pu-ad} that $d_v > 0$ for all $v \in \cC_\ad$. Thus $d_v \theta_v =1$ for all $v \in H$,  and so $\cC_\ad$ is Tannakian by Corollary \ref{cor:Tan_cri}, and braided equivalent to $\Rep(\hat{H})$.
\end{proof}

\begin{thm} \label{t:iso_adj}
Let $\cC$ be a modular fusion category.	Then the following statements are equivalent: 
\begin{enumerate}
    \item $\cC_{\ad}$ is isotropic;
    \item $\cC$ is braided equivalent to a modular fusion subcategory of $\zdr(\Vec_{L}^\w)$ for some finite abelian group $L$ and 3-cocycle $\w \in Z^3(L, \bC^\times)$;
    \item  $\cC_\ad$ is pointed, Tannakian,  and 
$$
\zdr(\Vec_{G}^\w) \stackrel{\br \o}{\simeq}\CC\boxtimes(\CC_A^0)^{\rev}\,,
$$
where $A$ is the regular algebra of $\cC_{\ad}$ and $G=\irr(\cC_{\pt})$;
    \item $\cC$ is a minimal modular extension of a pointed ribbon fusion category with isotropic M\"uger center.
\end{enumerate}
\end{thm}
\begin{proof} (a) $\Rightarrow$ (c)
Assume $\cC_{\ad}$ is isotropic. Let $G$ be group of invertible objects of $\cC$. Then by Lemma \ref{l:integral}, $\CC$ is pseudounitary  and $\cC_{\ad} \simeq \Rep(\hat H)$ as braided fusion categories, where $H=\{h \in G\mid \beta_{g,h}\beta_{h,g} =\id_{hg} \text{ for all }g \in G\}$. In view of Corollary \ref{cor:pu-ad}, we may equip $\cC$ with the canonical spherical structure. The regular algebra $A$ of $\cC_{\ad}$ is a connected \'etale algebra in $\cC_{\ad}$ (see \cite[Sec.~3]{DMNO} for details), and $\FPdim(A) = |H| = \FPdim(\cC_{\ad})$. Let $\CC_A$ denote the category of right $A$-modules in $\cC$. Then
$$
\FPdim(\CC_{A}) = \FPdim(\cC)/\FPdim(A) = |G|.
$$
Note that $\CC$ is faithfully $\hat{G}$-graded 
$
\CC = \bigoplus_{\chi\in \hat{G}} \CC_{\chi} 
$, 
where 
\[\irr(\CC_\chi) = \{x \in \irr(\cC)\mid \chi(g) = \frac{S_{x,g}}{d_x} \text{ for all } g \in G \}\,.\]
In particular,  $\CC_{\e} = \CC_{\ad}$, where $\e$ is the trivial character on $G$, and $A \in \CC_\e$.

Recall that $-\o A$ is the left adjoint of the forgetful functor $F: \CC_A \to \CC$. 
For any $x \in \CC_\chi$, $x\ot A \in \CC_\chi$, and let $M_\chi \in \CC_A$ be a simple $A$-submodule of $x \o A$. Obviously, $F(M_\chi) \in \CC_\chi$, and so $\{M_\chi\}_{\chi \in \hat{G}}$ is a set of nonisomorphic simple $A$-modules. Let $\FPdim_A(M_\chi)$ denote the Frobenius-Perron dimension of $M_\chi$ in $\CC_A$. Then we have
$$
|G|= \FPdim(\CC_A) \ge \sum_{\chi \in \hat{G}}\FPdim_A(M_\chi)^2 \ge |G|
$$ 
which forces $\FPdim_A(M_\chi) = 1$ for $\chi \in \hat{G}$ and $\irr(\CC_A) =\{M_\chi\mid \chi \in \hat{G}\}$. Therefore, $\CC_A \stackrel{\ot}{\simeq} \Vec_{\hat G}^\w$ for some $\w \in Z^3(\hat G,\bC^\times)$. Hence, by \cite[Cor.~3.30]{DMNO}, we have a braided equivalence
\[\zdr(\Vec_{\hat{G}}^{\w}) \cong \zdr(\cC_A) \cong \cC \bt (\cC_A^{0})^{\rev}\]
where $\cC_A^0$ is the category of local A-modules in $\cC$. Now (c) follows as $\hat{G} \cong G$ as groups.

(c) obviously implies (b),
% by Lemma \ref{l:integral}, [By def?]
and so we proceed to prove (b) $\Rightarrow$ (a). Suppose $\cC$ is braided equivalent to a fusion subcategory of $\zdr(\Vec_{L}^\w)$ for some finite abelian group $L$ and  $\w \in Z^3(L, \bC^\times)$. Then $\CC$ is integral and hence pseudounitary. By Corollary \ref{cor:pu-ad}, we may 
endow $\cC$ and $\zdr(\Vec_{L}^\w)$ with the canonical spherical structure, and the associated ribbon structures. Since $L$ is abelian, $\irr(\Rep(L))=\hat{L}$ can be identified as a subgroup of the group $U$ of invertible objects of $\zdr(\Vec_{L}^\w)$. Then, 
$$
\cC_{\ad} \subseteq (\zdr(\Vec_{L}^\w))_{\ad} = \cZ_2(U, \zdr(\Vec_{L}^\w)) \subseteq \cZ_2(\hat{L}, \zdr(\Vec_{L}^\w)) = \Rep(L) 
$$
which is isotropic. Therefore, $\cC_{\ad}$ is isotropic. 

By Lemma \ref{l:integral}, (a) implies (d). Conversely, let $\cC$ be a minimal modular extension of a pointed ribbon fusion category $\cC(G, q)$ with isotropic M\"uger center, then 
\[\cC_{\ad} = \cZ_2(\cC_\pt, \cC) \subset \cZ_2(G, \cC) = \cZ_2(G, \cC(G, q))\] 
where the third equality follows from the definition of minimal modular extension. Therefore, $\cC_{\ad}$ is isotropic.
\end{proof}

\begin{remark}
(i) The modular category $(\CC_A^0)^{\rev}$ in Theorem \ref{t:iso_adj} was completely identified in its proof. Since $\cC_\pt$ is a ribbon category, $\cC_\pt = \CC(G,q)$ where the quadratic form $q: G \to \bC^\times$ is given by $q(g) = \theta_g$ for $g \in G$.  The quadratic form $q$ induces a nondegenerate quadratic form, again denoted by $q$, on the quotient $G/H$ as $H$ is a maximal isotropic subgroup. The modular category $(\cC_A^0)^{\rev}$ is braided equivalent to $\CC(G/H,\ol q)$, where $\ol q$ is the complex conjugation of $q$.  

(ii) There exist some nonabelian groups $G$ and 3-cocycles $\w \in Z^3(G, \BC^\times)$ such that $\zdr(\Vec_G^\w)$ have isotropic adjoints as long as $\Vec_G^\w$ is Morita equivalent to  $\Vec_L^{\w'}$ for some abelian group $L$. For example, there exist some 3-cocycles $\w$ of any extraspecial 2-groups $E$ such that $\zdr(\Vec_E^\w)_\ad$ is isotropic  (cf. \cite[Thm. 4.6 and Thm. 5.2]{GMN07}).
\end{remark}

\section{2-Torelli triviality and 2-Engel groups}
In this section, we characterize the finite groups $G$ such that the MFC $\zdr(\Vec_G)$ is $2$-Torelli trivial and its adjoint is not isotropic
in Theorem \ref{t:Bn3}. First, we determine the groups $G$ such that $(\zdr(\Vec_G))_{\ad}$ is not isotropic in Proposition \ref{p:4.1}.  Second, in Proposition \ref{p:4.2}, we identify the groups $G$ with $\zdr(\Vec_G)$ being 2-Torelli trivial are exactly the 2-\emph{Engel} groups (cf. \cite[Sec. 12.3]{Rob}), that means
\begin{equation} \label{eq:2-engel}
    [[x,y],y] = 1 \qquad \text{for all }x,y \in G,
\end{equation}
where $[x,y]=x^{-1} y^{-1} xy$.

By using GAP with the conditions given by Propositions \ref{p:4.1} and \ref{p:4.2}, we find that the smallest order of these groups is $3^7$ and all these smallest examples are semidirect products of an extraspecial 3-group of order $3^3$ and some of its faithful indecomposable modules of order $3^4$. They are
\begin{equation}
B(3,3) = \BZ_3^4 \rtimes H_{27}, \quad (\BZ_9 \oplus \BZ_3^2)\rtimes H_{27}, \quad \text{and} \quad   \BZ_9^2\rtimes M_{27}
\end{equation}
where $H_{27}$ is the Heisenberg group, $M_{27}$ is the exponent 9 extraspecial 3-group of order 27, and $B(3,3)$ denotes the rank 3 Burnside group of exponent 3.

Let $B(3,r)$ be the Burnside group of exponent 3 and with $r$ generators \cite[18.2]{Hall}, which can be defined  by the presentation
$$
B(3,r) = \langle x_1, \dots, x_r  \mid w^3=1 \text{ for every word  $w$ formed by }x_1, \dots, x_r\rangle\,.
$$
 Then $B(3,1)$ is the cyclic group of order 3, and $B(3, 2)$ is the Heisenberg group $H_{27}$ of order \(27\).  We prove in Theorem \ref{t:Bn3} that the infinite family $\{B(3,r): r \ge 3\}$ satisfies the aforementioned conditions. We first establish the necessary and sufficient conditions each of these aforementioned conditions. 

 For the reminder of this section, we simply assume the canonical spherical structure of $\zdr(\Vec_G)$ for any finite group $G$.

\begin{prop} \label{p:4.1}
For any finite group $G$,   $\zdr(\Vec_G)$ admits an isotropic adjoint subcategory if and only if $G' \subseteq Z(G)$. Equivalently, $(\zdr(\Vec_G))_{\ad}$ is not isotropic, if and only if $G' \not\subset Z(G)$.
\end{prop}
\begin{proof}
Let $\cC = \zdr(\Vec_G)$ and $Z = Z(G)$. We first show that $\CC$ is a minimal modular extension of $\CC_\pt$ if and only if $G' \subseteq Z$. The category $\cC_\pt$ is determined by the quadratic form $\theta: A \to \BC^\times$, where $\theta(V) = \theta_V$, where $A$ denotes the group $\Inv(\CC)$ of invertible objects of $\CC$. We let $B = \Inv(\cZ_2(\CC_\pt))$. By \cite[Prop. 5.2 and Eqn. (5.30)]{MN14}, 
the quadratic form $(A, \theta)$ is equivalent to   $(Z \times \hat G, q)$ where $q(z, \chi) = \chi(z)$. 
 Thus, $B\cong (Z \times \hat G)^\perp$, the radical of the  associated bicharacter $b$ of the quadratic form 
$(Z \times \hat G, q)$. The bicharacter $b$ of $(Z \times \hat G, q)$ is given by $b((z,\chi), (z', \chi')) = \chi(z')\chi'(z)$. Thus, $(z, \chi) \in  (Z \times \hat G)^\perp$ if and only if $b((z', \chi'), (z,\chi))=1$ for all $z' \in Z$ and $\chi \in \hat{G}$, which is equivalent to $z \in G'$ and $\chi|_{Z} = 1$. In particular, $q(z,\chi) = \chi(z) = 1$ for any $(z,\chi) \in (Z \times \hat G)^\perp$. Therefore, $\cZ_2(\CC_\pt)$ must be isotropic (or Tannakian). 

The group of 1-dimensional characters $\chi$ of $G$ satisfying $\chi|_{Z}=1$ is $\widehat{G/Z}$. Therefore,
$$
|B| = |Z \cap G'||\widehat{G/Z}| = |Z \cap G'|\left|\frac{G}{G'Z}\right| = \frac{|Z \cap G'||G|}{|G'||Z|/|G'\cap Z|} =\frac{|G||Z \cap G'|^2}{|G'||Z|} 
$$ 
and so
$$
|A||B| = |Z||\hat G| \cdot\frac{|G||Z \cap G'|^2}{|G'||Z|} = \frac{|G|^2 |Z \cap G'|^2}{|G'|^2}\,.
$$
Thus, $\CC$ is a minimal modular extension of $\CC_\pt$ if and only if
$$
|G|^2=|A||B| = \frac{|G|^2 |Z \cap G'|^2}{|G'|^2}
$$
or equivalently
$$
|G'|=|Z \cap G'|\,.
$$
This equality holds if and only if $G' \subseteq Z$.

By Theorem \ref{t:iso_adj} (d),  if $\CC$ is a minimal modular extension of $\CC_\pt$, then $\CC$ admits an isotropic adjoint. Conversely,  $\CC$ has an isotropic adjoint, then $\CC$ is a minimal modular extension of $\CC_\pt$ by Lemma \ref{l:integral}.  The proposition now follows the first claim.
\end{proof}

\begin{remark}
Finite groups $G$ satisfying the condition $G' \subseteq Z(G)$ are exactly nilpotent groups of class $\le 2$.  Proposition \ref{p:4.1} is a characterization of nilpotent groups $G$ of class $\le 2$ as those satisfying $\zdr(\Vec_G)_{\ad}$ being isotropic.  
\end{remark}

Recall that the center $\zdr(\Vec_G)$ is equivalent to $\Rep(D(G))$ or $\Rep(D(G))^{\rev}$ as braided fusion category, where $D(G)$ denotes the Drinfeld double of $\BC G$. These two possibilities completely depend on the choice of center construction $\zdr(\Vec_G)$, and the ribbon structures of these two choices are inverse of each other. Since we are interested the condition $\theta_{I} = \id$, which is equivalent to  $\theta^{-1}_{I} = \id$. We simply assume $\zdr(\Vec_G) \stackrel{br\,\o }{\simeq} \Rep(D(G))$. 

Following the convention of $D(G)$ in  \cite[IX.4.3]{Kas95}, the ribbon structure $\theta$ of $\Rep(D(G))$ is given by $\theta_V = u^{-1}|_V$ for any $V \in \Rep(D(G))$, where $u$ is the Drinfeld element of $\Rep(D(G))$. Recall from \cite[IX.4.3]{Kas95} (or \cite[10.3.4]{Mont}), $D(G) = \BC G^* \o \BC[G]$ as $\BC$-linear space. Let $e_g \in \BC G^*$ denote the delta function at \(g\), and write
\[
e_g\,x:=e_g\otimes x\qquad\text{ for } g,x\in G.
\]
Using the shorthand $g^x = x^{-1} g x$, the product, coproduct $\Delta$ and antipode $S$ of $D(G)$ are given by
\[
 e_g\,x \cdot e_h\,y=\delta_{g^x,h}e_g\,xy,
 \qquad
 \Delta(e_g x)=\sum_{ab=g}e_a\,x \o e_b\, x, \qquad S(e_g x)=e_{(g^{-1})^x} \,x^{-1}.
\]
The Drinfeld element $u \in D(G)$  is given by
\[
    u = \sum_{{y} \in G} e_y y^{-1} \quad \text{and so}\quad u^{-1} = \sum_{{y} \in G} e_y\, y\,. 
\]
The adjoint representation $D(G)_{\ad}$ is simply  $D(G)$ equipped with the adjoint $D(G)$-action $\triangleright$ defined by using the   Sweedler notation $\Delta(a) = \sum_{(a)} a_1 \o a_2$ as
\[
    a \triangleright b := \sum_{(a)} a_1 b S(a_2) \qquad\text{for } a,b \in D(G) 
\]
In particular, $D(G)_{\ad} \cong I = \sum_{X \in \Irr(D(G))} X \o X^*$.
\begin{prop}\label{p:4.2}
Let $G$ be a finite group. Then the following conditions are equivalent:
\begin{enumerate}
    \item The twist of $D(G)_{\ad}$ in $\Rep(D(G)))$ is the identity;
    \item The twist of $I$ in $\zdr(\Vec_G)$ is the identity;
    \item $Z(\Vec_G)$ (or $\Rep(D(G))$) is 2-Torelli trivial;
    \item $G$ is a 2-Engel group, i.e., $G$ satisfies \eqref{eq:2-engel}.
\end{enumerate}  
\end{prop}
\begin{proof}
In view of the preceding discussions and Lemma \ref{lem:g2}, it suffices to prove the equivalence of (a) and (d) by considering action of $u^{-1}$ on $D(G)_{\ad}$.  For any $g,x \in G$,
\begin{align*}
u^{-1} \triangleright (e_g x) 
= & {\sum_{y}} \sum_{ab=y} e_a y \cdot e_g x \cdot S(e_b y) 
= {\sum_{y}}\sum_{ab=y} \delta_{a^y, g} \cdot e_a \, yx \cdot e_{(b^{-1})^y}\, y^{-1}\\
 =&{\sum_{y}}\sum_{ab=y} \delta_{a^y, g}\delta_{a^{yx}, (b^{-1})^y} e_{a}\, x^{y^{-1}} = \sum_y \delta_{g^x, y^{-1} g} e_{g^{y^{-1}}}\, x^{y^{-1}}  = e_{g^z} x^z\
\end{align*}
where $z = [x, g^{-1}]$. Thus, $u^{-1}$ acts trivially on $D(G)_{\ad}$ if and only if $g = g^{ [x, g^{-1}]}$ and $x = x^{ [x, g^{-1}]}$ for all $x, g \in G$. The last two conditions are equivalent to that $[y,x]$ commutes with $x$ for all $x, y \in G$ or that $G$ is a 2-Engel group by definition.
\end{proof}
\begin{remark} \label{r:2-engel}
By \cite[12.3.6]{Rob}, every 2-Engel group $G$ is nilpotent of class $\le 3$, and $[G',G]$ has exponent 3.
\end{remark}
\begin{thm} \label{t:Bn3} The finite groups $G$ satisfying both
\begin{enumerate}
\item[\rm (i)] $I \in \zdr(\Vec_G)_{\ad}$ has trivial twists, and 
\item[\rm (ii)] $\zdr(\Vec_G)_{\ad}$ is not isotropic, 
\end{enumerate}
are precisely the 2-Engel groups of nilpotent class 3. 

In particular, $\{B(3,r) \mid r \ge 3\}$ is an infinite family of finite groups satisfying these conditions.
\end{thm}
\begin{proof} Recall from Remark \ref{r:2-engel} that 2-Engel groups are  nilpotent of class $\le 3$. A 2-Engel group of nilpotent class 3 is obviously not nilpotent class 2 by definition and hence $G' \not\subseteq Z(G)$. So the first statement follows immediately from Propositions \ref{p:4.1} and \ref{p:4.2}. 

According to \cite[18.2]{Hall}, $B(3,r)$ is a 2-Engel for $r \ge 1 $. Moreover, for $r \ge 3$, every element in $B(3,r)$ can be written uniquely as a product of distinct commutators of lengths 1, 2 and 3. The subgroup $[G', G]$ consists of the products of these length 3 commutators and its order is $3^{\binom{r}{3}}$. In particular, $B(3,r)$ has nilpotent class 3. 
\end{proof}
\begin{remark}
Since $B(3,2) = H_{27}$ and $H_{27}' = Z(H_{27})$, the adjoint of $\zdr(\Vec_{B(3,2)})$ is isotropic by Proposition \ref{p:4.1}.
\end{remark} 
\begingroup
\small
\bibliographystyle{abbrv}
\bibliography{ref}
\endgroup

\authorinfo

\end{document}

%% file: Macros.tex
\usepackage{amsmath}
\usepackage{amsfonts,amssymb,amsthm,color,dsfont,amsbsy}
\usepackage{mathrsfs}

\usepackage[T1]{fontenc}
\usepackage{mlmodern} 

\usepackage{caption}
\makeatletter
\renewcommand{\@seccntformat}[1]{\csname the#1\endcsname.\ }  % add dot visually
\makeatother

\makeatletter
\def\numberline#1{\hb@xt@\@tempdima{#1.\hfil}}
\makeatother
\usepackage[dvipsnames]{xcolor} 
\definecolor{bleudefrance}{rgb}{0.19, 0.35, 0.9}

\usepackage{geometry}
\usepackage{hyperref}
\hypersetup{
  colorlinks = true,
  allcolors = blue,
}

\usepackage[all, knot, color]{xy}

\usepackage{amscd}
\usepackage{mathtools}
\usepackage{bm}
\usepackage{eucal}
\usepackage{extpfeil}
\usepackage{float}

\usepackage{tikz}
\usetikzlibrary{cd, calc, knots}
\usetikzlibrary{shapes.geometric}
\usetikzlibrary{decorations, decorations.markings} 
\usetikzlibrary{arrows, arrows.meta}
\tikzset{
    partial ellipse/.style args={#1:#2:#3}{
        insert path={+ (#1:#3) arc (#1:#2:#3)}
    }
}

\tikzset{->-/.style={decoration={
markings,
mark=at position #1 with {\arrow{>}}},postaction={decorate}}}
\tikzset{-<-/.style={decoration={
markings,
mark=at position #1 with {\arrow{<}}},postaction={decorate}}}
\tikzset{-<<-/.style={decoration={
markings,
mark=at position #1 with {\arrow{-latex}}},postaction={decorate}}}

\usepackage{graphicx}

\usepackage{enumitem}
\setlist[enumerate,1]{label={\rm (\alph*)}, topsep=.2em, noitemsep=1em, leftmargin=35pt, labelsep=7pt}
\newtheorem{thm}{Theorem}[section]

\newtheorem{cor}[thm]{Corollary}

\newtheorem{prop}[thm]{Proposition}

\newtheorem{lemma}[thm]{Lemma}

\theoremstyle{definition}
\newtheorem{exmp}[thm]{Example}
\newenvironment{example}{\pushQED{\qed}\exmp}{\popQED\endexmp}

\newtheorem{defn}[thm]{Definition}
\newenvironment{definition}{\pushQED{\qed}\defn}{\popQED\enddefn}
\newtheorem{rmk}[thm]{Remark}
\newenvironment{remark}{\pushQED{\qed}\rmk}{\popQED\endrmk}

\numberwithin{equation}{section}

\def\cC{\mathcal{C}}
\def\cD{\mathcal{D}}
\def\cE{\mathcal{E}}
\def\cF{{\mathcal{F}}}

\def\cZ{\mathcal{Z}}
\def\bC{{\mathbb{C}}}

\def\bN{{\mathbb{N}}}

\def\bR{{\mathbb{R}}}
\def\bS{{\mathbb{S}}}

\def\bZ{{\mathbb{Z}}}
\def\ld{{\lambda}}

\def\be{{\beta}}
\def\e{{\varepsilon}}

\def\w{{\omega}}
\def\om{{\omega}}

\def\o{{\,\otimes\,}}
\def\s{{\sigma}}

\def\fA{\mathfrak{A}}

\def\fN{{\mathfrak{N}}}

\def\smat{{\begin{pmatrix}0&-1\\1&0\end{pmatrix}}}
\def\tmat{{\begin{pmatrix}1&1\\0&1\end{pmatrix}}}
\def\1{{\mathds{1}}}
\def\to{\rightarrow}

\def\Mp2{{\operatorname{Mp}_2(\BZ)}}
\def\PSU{{\operatorname{PSU}}}

\def\Sp{{\operatorname{Sp}}}
\def\op{{\operatorname{op}}}
\def\pt{{\operatorname{pt}}}
\def\ad{{\operatorname{ad}}}
\def\br{{\operatorname{br}\,}}

\def\Aut{{\operatorname{Aut}}}

\def\Rep{{\operatorname{Rep}}}

\newcommand\Irr{{\operatorname{Irr}}}
\def\irr{{\operatorname{Irr}}}

\def\id{{\operatorname{id}}}

\def\Inv{{\operatorname{Inv}}}

\def\tr{{\operatorname{tr}}}

\def\ord{{\operatorname{ord}\,}}

\def\GL{{\operatorname{GL}}}

\def\Vec{{\operatorname{Vect}}}

\def\diag{{\operatorname{diag}}}

\def\ev{{\operatorname{ev}}}
\def\coev{{\operatorname{coev}}}

\renewcommand\subset{ \subseteq }

\usepackage{diagbox}
\usepackage{graphics}

\newcommand\ol[1]{\overline{#1}}

\newcommand{\bt}{\boxtimes}

\newcommand{\ot}{\otimes}

\newcommand{\FPdim}{\operatorname{FPdim}}

\def\op{{\operatorname{op}}}

\def\x{\times}

\def\SL{{\operatorname{SL}}}

\usepackage{leftidx}
\usepackage{tcolorbox}

\newcommand{\rev}{\operatorname{rev}}
\renewcommand{\ker}{\operatorname{Ker}}

\newcommand{\FP}{\FPdim}

\newcommand{\BC}{\mathbb{C}}

\newcommand{\BZ}{\mathbb{Z}}

\newcommand{\CC}{{\mathcal{C}}}
\newcommand{\DD}{\mathcal{D}}

\newcommand{\FF}{\mathcal{F}}

\renewcommand{\ker}{\operatorname{ker}}

\def\o{\otimes}
\def\rev{\operatorname{rev}}
\def\w{\omega}

\newcommand{\PGL}{\operatorname{PGL}}

\newcommand{\MCG}[1]{\operatorname{Mod}(\Sigma_{#1})}
\newcommand{\mcg}[1]{\operatorname{Mod}(\Sigma_{#1})}
\newcommand{\torelli}[1]{\mathscr{T}_{#1}}
\def\zmug{\cZ_{2}}
\def\zdr{{\CMcal{Z}}}

\def\isotopy{\operatorname{isotopy}}

\def\ITwoGen{\varpi}
\def\dt{{\CMcal{T}}}
\def\db{{\CMcal{B}}}
\def\cmcs{{\CMcal{S}}}

\newcommand{\RT}{\operatorname{RT}}

\newcommand{\RTspace}[2]{\RT_{#1}(\Sigma_{#2})}

\newcommand{\RTmat}[2]{O_{{#1}, {#2}}}
\newcommand{\RTrep}[2]{\rho_{{#1}, {#2}}}

\newcommand{\ncg}[2]{n_{{#1}, {#2}}}

\newcommand{\AdSet}[1]{I^{(#1)}}

\newcommand{\hmu}{\hat{\mu}}
\newcommand{\cmu}{\check{\mu}}
\newcommand{\Wbase}[4]{{\hmu_{#1}(\bm{#2}, \bm{#3}, \bm{#4})}}
\def\cyl{\operatorname{cyl}}

\def\ot{{\otimes}}

\def\Aut{{\operatorname{Aut}}}

\def\e{{\varepsilon}}
\def\piv{\varkappa}
\def\PGLC{{\PGL(\ncg{\cC}{g}, \bC)}}
\def\Spg{{\Sp(2g,\bZ)}}
\newcommand{\DualDef}{
\coev_x = \vcenter{\hbox{\scalebox{0.6}{
\begin{tikzpicture}[line width=2pt]
\draw[->-=.5] (0,0) arc(0:180:1);
\node[below] at (0,0) {\Large $x^*$};\node[below] at (-2,0) {\Large $x$};
\end{tikzpicture}}}}: \1 \to x \ot x^{*}
\quad\text{and}\quad
\ev_x= \vcenter{\hbox{\scalebox{0.6}{
\begin{tikzpicture}[line width=2pt]
\draw[->-=.5] (0,0) arc(0:-180:1);
\node [above] at (0,0) {\Large $x$}; \node [above] at (-2,0) {\Large $x^*$};
\end{tikzpicture}}}}: x^* \ot x \to \1}
\newcommand{\Brd}{
\vcenter{\hbox{\scalebox{0.6}{
\begin{tikzpicture}[line width=2pt]
\begin{knot}  
\strand[out=down,in=up] (0,0) to (1.3,-2);
\strand[out=down,in=up] (1.3,0) to (0,-2);
\node[above] at (0,0) {\Large  $x$};
\node[above] at (1.3,0) {\Large  $y$};
\end{knot}
\end{tikzpicture}}}}}
\newcommand{\BrdInv}{\vcenter{\hbox{\scalebox{0.6}{
\begin{tikzpicture}[line width=2pt]
\begin{knot}[yscale=-1]  
\strand[out=down,in=up] (0,0) to (1.3,-2);
\strand[out=down,in=up] (1.3,0) to (0,-2);
\node[above] at (0,-2) {\Large  $y$};
\node[above] at (1.3,-2) {\Large  $x$};
\end{knot}
\end{tikzpicture}}}}}
\newcommand{\TrivalentVertexCheck}{
\vcenter{\hbox{\scalebox{0.6}{
\begin{tikzpicture}[line width=2pt]
\draw (0,0) -- ++(45:1.5)
      (0,0) -- ++(135:1.5)
      (0,0) -- (0,-1);
\draw ([shift=(45:0.3cm)]0,0) arc (45:135:0.3cm);
\fill[black] (0,0) circle (3pt);
\node[above] at (0,0.3) {\Large $j$};
\node[above] at (-1.2, 1) {\Large $b$};
\node[above] at (1.2, 1) {\Large $c$};
\node[below] at (0, -1) {\Large $a$};
\end{tikzpicture}}}}}
\newcommand{\BranchDown}[1]{
\begin{scope}[shift={#1}]
\draw (0,0)--(-1,-1); \draw (0,0)--(1,-1);
\draw ([shift=(225:0.3cm)]0,0) arc (225:315:0.3cm);
\fill[black] (0,0) circle (3pt);
\end{scope}}
\newcommand{\TrivalentVertexHat}{
\vcenter{\hbox{\scalebox{0.6}{
\begin{tikzpicture}[line width=2pt]
\draw (0,0) -- ++(225:1.5)
      (0,0) -- ++(315:1.5)
      (0,0) -- (0,1);
\draw ([shift=(225:0.3cm)]0,0) arc (225:315:0.3cm);
\node[below] at (0,-0.3) {\Large $j$};
\node[below] at (-1.2, -1) {\Large $b$};
\node[below] at (1.2, -1.1) {\Large $c$};
\node[above] at (0, 1) {\Large $a$};
\fill[black] (0,0) circle (3pt);
\end{tikzpicture}}}}}
\newcommand{\StraightLine}[1]{
\vcenter{\hbox{\scalebox{0.6}{
\begin{tikzpicture}[line width=2pt]
\draw (0,0) -- (0,3);
\node[above] at (0,3) {\Large $#1$};
\end{tikzpicture}}}}}
\newcommand{\CupCap}{
\vcenter{\hbox{\scalebox{0.6}{
\begin{tikzpicture}[line width=2pt]
\draw (0,0) -- ++(45:1.5)
      (0,0) -- ++(135:1.5)
      (0,0) -- (0,-1);
\draw ([shift=(45:0.3cm)]0,0) arc (45:135:0.3cm);
\node[above] at (0,0.3) {\Large $j$};
\node[above] at (-1.2, 1) {\Large $b$};
\node[above] at (1.2, 1) {\Large $c$};
\fill[black] (0,0) circle (3pt);
\node[right] at (0,-0.5) {\Large $a$};
\draw (0,-1) -- ++(225:1.5)
      (0,-1) -- ++(315:1.5)
      (0,-1) -- (0,0);
\draw ([shift=(225:0.3cm)]0,-1) arc (225:315:0.3cm);
\node[below] at (0,-1.3) {\Large $j$};
\node[below] at (-1.2, -2) {\Large $b$};
\node[below] at (1.2, -2.13) {\Large $c$};
\fill[black] (0,-1) circle (3pt);
\end{tikzpicture}}}}}
\newcommand{\CapCup}{
\vcenter{\hbox{\scalebox{0.6}{
\begin{tikzpicture}[line width=2pt]
\draw (0,0) -- ++(45:1.5) -- ++(135:1.5) node(a1){} 
        -- ++(90:1) node(a2){}
      (0,0) -- ++(135:1.5) -- ++(45:1.5)
      (0,0) -- (0,-1);
\draw ([shift=(45:0.3cm)]0,0) arc (45:135:0.3cm);
\draw (a1) +(225:0.3cm) arc (225:315:0.3cm);
\node[above] at (0,0.3) {\Large $k$};
\node[above] at (-1.2, 1) {\Large $b$};
\node[above] at (1.2, 1) {\Large $c$};
\fill[black] (0,0) circle (3pt);
\fill[black] (a1) circle (3pt);
\node[below] at (0,-1) {\Large $a$};
\node[above] at (a2) {\Large $a$};
\node[below] at ($(a1) + (0,-0.3)$) {\Large $j$};
\end{tikzpicture}}}}}
\newcommand{\HoleWithCurve}[1]{
\begin{scope}[shift={#1}]
\hole{0,-0.5};
\draw[purple] (0, 0.21) arc (-90:90:0.3 and 0.87);
\draw[purple, dashed] (0, 0.21) arc (270:90:0.3 and 0.87);
\draw[blue] (0,0) ellipse (1.2 and 0.6);
\end{scope}}
\newcommand{\LeftEquator}[1]
{\begin{scope}[shift={#1}]
\draw[Green] (0.48,0) arc (180:90:1.52 and 0.6);
\draw[Green, dashed] (0.48,0) arc (180:270:1.52 and 0.6);    
\end{scope}}
\newcommand{\RightEquator}[1]{
\begin{scope}[shift={#1}]
\draw[Green] (-0.48,0) arc (0:90:1.52 and 0.6);
\draw[Green, dashed] (-0.48,0) arc (0:-90:1.52 and 0.6);        
\end{scope}}
\renewcommand{\hole}[1]{
\draw ([shift=(45:0.707)]$(#1)$) arc (45:135:0.707);
\draw ([shift=(210:0.707)]$(#1)+(0,1)$) arc (210:330:0.707)
}

\newcommand{\grayhole}[1]{
\draw[gray] ([shift=(45:0.707)]$(#1)$) arc (45:135:0.707);
\draw[gray] ([shift=(210:0.707)]$(#1)+(0,1)$) arc (210:330:0.707)
}